\pdfoutput=1

\documentclass[11pt]{article}
\usepackage[utf8]{inputenc}

\usepackage{amssymb}
\usepackage{mathtools}
\usepackage{mathrsfs}
\usepackage[usenames,dvipsnames]{xcolor}

\usepackage[numbers,comma,square]{natbib}
\setcitestyle{notesep={:~}}
\usepackage{hyperref}
\hypersetup{colorlinks=true,linkcolor={Brown},citecolor={Brown},urlcolor={Brown}}
\usepackage[explicit]{titlesec}
\titleformat{\subsubsection}[runin]{\bfseries} {\thesubsubsection.} {5pt} {#1} [\hspace*{.2cm}]%
\usepackage[compress]{cleveref}
\usepackage{amsthm}
\usepackage{thmtools, thm-restate}
\declaretheoremstyle[%
  spaceabove=.8\baselineskip,%
  spacebelow=.8\baselineskip,%
  headfont=\bfseries,%
  notefont=\normalfont,%
  bodyfont=\itshape,%
  postheadspace=.5em%
]{thms}

\usepackage{url}

\declaretheoremstyle[%
  spaceabove=.8\baselineskip,%
  spacebelow=.8\baselineskip,%
  headfont=\bfseries,%
  notefont=\normalfont,%
  bodyfont=\normalfont,%
  postheadspace=.5em%
]{defn}

\numberwithin{equation}{section}%
\theoremstyle{thms}
\newtheorem{thrm}[equation]{Theorem}
\newtheorem{cor}[equation]{Corollary}
\newtheorem{lem}[equation]{Lemma}
\newtheorem{prop}[equation]{Proposition}

\theoremstyle{defn}
\newtheorem{cns}[equation]{Construction}
\newtheorem{defn}[equation]{Definition}
\newtheorem*{defn*}{Definition}
\newtheorem{exa}[equation]{Example}

\newtheorem{covn}[equation]{Convention}
\newtheorem{notn}[equation]{Notation}
\newtheorem{rmk}[equation]{Remark}

  \crefname{prop}{Proposition}{Propositions}
  \crefname{cor}{Corollary}{Corollaries}
  \crefname{lem}{Lemma}{Lemmata}
  \crefname{cns}{Construction}{Constructions}
  \crefname{defn}{Definition}{Definitions}
  \crefname{rmk}{Remark}{Remarks}
  \crefname{covn}{Convention}{Conventions}
  \crefname{notn}{Notation}{Notations}
  \crefname{exa}{Example}{Examples}
  \crefname{thrm}{Theorem}{Theorems}
  \crefname{thrmprime}{Theorem'}{Theorem'}
\usepackage[shortlabels]{enumitem}
  \newlist{axenum}{enumerate}{1} %
  \setlist[axenum]{label=(\Alph*)}
  \crefname{axenumi}{Axiom}{Axioms}
\setlist{itemsep=.1em}

\usepackage{bm}
\usepackage[full]{textcomp} %
\usepackage[p,osf]{fbb} %
\usepackage[scaled=.95,type1]{cabin} %
\usepackage[varqu,varl]{zi4}%
\usepackage[T1]{fontenc} %
\usepackage[libertine]{newtxmath}
\usepackage[cal=boondoxo,bb=boondox,frak=boondox]{mathalfa}
\usepackage{amsmath}

\usepackage{tikz}
\usetikzlibrary{arrows,shapes,positioning,backgrounds,cd}
\tikzset{-,>=stealth',shorten >=2pt,shorten <=2pt,
  main node/.style={circle,fill=blue!20,inner sep=2pt,font=\sffamily\tiny\bfseries},
  desc/.style={font=\sffamily\footnotesize, align=center}
}
\usepackage{subcaption}
\usepackage{textcomp}
\usepackage{tikz-cd}
\RequirePackage[margin=3.3cm]{geometry}
\usetikzlibrary{arrows,arrows.meta}
\tikzcdset{arrow style=tikz}

\newcommand\footnotewomarker[1]{%
  \begingroup
  \renewcommand\thefootnote{}\footnote{#1}%
  \addtocounter{footnote}{-1}%
  \endgroup
}
\usepackage{xspace}
\newcommand{\icat}{$\infty$\nobreakdash-category\xspace}
\newcommand{\icats}{$\infty$\nobreakdash-categories\xspace}
\newcommand{\icategorical}{$\infty$\nobreakdash-catego\-ri\-cal\xspace}
\newcommand{\ioperad}{$\infty$-operad\xspace}
\newcommand{\ioperads}{$\infty$-operads\xspace}
\newcommand{\subicat}{sub-$\infty$-category\xspace}

\newcommand{\iCat}{\textup{Cat}_{\infty}}

\newcommand{\Prst}{\textup{Pr}_{\textup{st}}}
\newcommand{\Prstomega}{\textup{Pr}_{\textup{st}}^{\omega}}

\DeclareMathOperator{\Hm}{H}
\DeclareMathOperator{\Hom}{Hom}

\DeclareMathOperator{\mapsp}{map}
\DeclareMathOperator{\Pic}{Pic}
\DeclareMathOperator{\M}{M}
\DeclareMathOperator{\Spec}{Spec}
\newcommand{\one}{\mathbb{1}}

\newcommand{\id}{\on{id}}
\newcommand{\into}{\hookrightarrow}

\newcommand{\calg}[1]{%
  \operatorname{CAlg}(#1)}
\newcommand{\calgnu}[1]{%
  \operatorname{CAlg_{nu}}(#1)}
\newcommand{\calgqu}[1]{%
  \operatorname{CAlg_{qu}}(#1)}

\usepackage{etoolbox}
\newcommand{\fun}[3][]{%
  \ifblank{#1}{
    \operatorname{Fun}(#2,#3)}
  {
    \operatorname{Fun}^{#1}(#2,#3)}}
\newcommand{\Mod}[2][]{%
  \ifblank{#1}{
    \operatorname{Mod}(#2)}
  {
    \operatorname{Mod}_{#1}(#2)}}

\providecommand{\aone}{\ensuremath{\mathbb{A}^{1}}}

\DeclareMathOperator{\opname}{op}
\DeclareMathOperator{\Ind}{Ind}

\newcommand{\FF}{\mathbb{F}}

\newcommand{\Q}{\mathbb{Q}}

\newcommand{\Qellbar}{\overline{\mathbb{Q}}_{\ell}}
\newcommand{\Qbar}{\overline{\mathbb{Q}}}
\newcommand{\CC}{\mathbb{C}}

\newcommand{\RR}{\mathbb{R}}

\newcommand{\Z}{\mathbb{Z}}

\newcommand{\E}[1]{\mathbb{E}_{#1}}
\newcommand{\affine}{\mathbb{A}}

\newcommand{\lcl}{\mathcal{L}}
\newcommand{\cD}{\mathcal{D}}
\newcommand{\et}{\mathrm{\acute{e}t}}
\newcommand{\Nis}{\mathrm{Nis}}
\newcommand{\an}{\mathrm{an}}
\newcommand{\mot}{\mathrm{mot}}
\newcommand{\ct}{\mathrm{ct}}

\newcommand{\on}[1]{\operatorname{#1}}
\newcommand{\op}{^{\opname}}
\newcommand{\lax}{\mathrm{lax}}
\newcommand{\coh}{\mathrm{coh}}
\newcommand{\BM}{\mathrm{BM}}

\providecommand{\Fun}{\ensuremath{\mathrm{Fun}}}
\DeclareMathOperator{\Homint}{\underline{Hom}}
\providecommand{\Fin}{\ensuremath{\mathrm{Fin}}}

\providecommand{\cofib}{\ensuremath{\mathrm{Cofib}}}

\providecommand{\Cat}[1]{\mathrm{Cat}_{\mathrm{#1}}}
\providecommand{\mCat}[1]{\mathrm{Cat}^{\otimes}_{\mathrm{#1}}}

\providecommand{\Tri}{\mathrm{Tri}}
\providecommand{\mTri}{\mathrm{Tri}^{\otimes}}

\newcommand{\CoSy}[1]{
  \operatorname{CoSy}_{#1}}

\providecommand{\COSY}[1]{\mathbf{CoSy}^{\textup{Pr}}_{\mathrm{#1}}}

\providecommand{\cCOSY}[1]{\mathbf{CoSy}^{\omega}_{\mathrm{#1}}}

\newcommand{\propsub}{\mathrm{prop}}

\DeclareMathOperator{\SH}{SH}
\DeclareMathOperator{\DA}{DA}
\DeclareMathOperator{\DM}{DM}
\DeclareMathOperator{\DTM}{DTM}

\DeclareMathOperator{\Shv}{Shv}

\DeclareMathOperator{\thom}{Th}
\DeclareMathOperator{\mthom}{MTh}
\DeclareMathOperator{\thomexp}{\ul{Th}}
\DeclareMathOperator{\mthomexp}{\ul{MTh}}

\providecommand{\Ga}[1]{\ifblank{#1}{\ensuremath{\mathbb{G}_{\mathrm{a}}}}{\ensuremath{\mathbb{G}_{\mathrm{a},#1}}}}

\providecommand{\FT}{\mathrm{FT}}

\newcommand{\Sch}[2][]{  \ifblank{#2}{\operatorname{Sch}^{#1}}{\operatorname{Sch}_{#1}^{\text{#2}}}}
\newcommand{\sch}[1][]{
  \ifblank{#1}{\operatorname{Sch}}{\operatorname{Sch}_{#1}}}
\newcommand{\schop}[1][]{
  \ifblank{#1}{\operatorname{Sch}^{\textup{op}}}{\operatorname{Sch}_{#1}^{\textup{op}}}}
\newcommand{\Sm}[1]{
    \operatorname{Sm}_{#1}}
\newcommand{\sm}[1]{
  \operatorname{Sm}^{\grp}_{#1}}
\newcommand{\smop}[1]{
  (\operatorname{Sm}^{\grp}_{#1})^{\opname}}

\newcommand{\Sp}{\mathrm{Sp}}

\newcommand{\fin}{\operatorname{Fin}_*}
\newcommand{\Comm}{\operatorname{Comm}^\otimes}
\newcommand{\Commnu}{\operatorname{Comm_{nu}}^\otimes}
\newcommand{\tnu}{\textup{nu}}
\newcommand{\isoto}{\overset{\sim}{\,\to\,}}
\newcommand{\isofrom}{\overset{\sim}{\,\leftarrow\,}}
\let\xfrom\xleftarrow
\let\xto\xrightarrow
\newcommand{\mc}[1]{\mathcal{#1}}
\newcommand{\cC}{\mc{C}}
\newcommand{\cE}{\mc{E}}

\newcommand{\cO}{\mc{O}}

\newcommand{\resp}[1]{%
  \textup{(}resp.\ #1\textup{)}\xspace}

\providecommand{\ep}[2][]{\ifblank{#1}{{#2}_{\mathrm{exp}}}{{#2}_{\mathrm{exp},#1}}}
\providecommand{\hep}[2][]{\ifblank{#1}{{\mathrm{h}#2}_{\mathrm{exp}}}{\mathrm{h}{#2}_{\mathrm{exp}}_{#1}}}

\providecommand{\h}[1]{\mathrm{h}#1}

\providecommand{\grp}{\mathbb{G}}
\providecommand{\plus}{\ensuremath{\mathrm{\mu}}}

\providecommand{\unit}{0}
\providecommand{\base}{B}

\providecommand{\conv}{\mathop{\scalebox{1.3}{\raisebox{-0.2ex}{$\ast$}}}}%
\providecommand{\convbox}{\circledast\boxtimes}%
\providecommand{\convsharp}{\conv_{\sharp}}%

\DeclareMathOperator{\Corr}{Corr}
\DeclareMathOperator{\corr}{corr}

\providecommand{\pot}{a}

\providecommand{\ul}{\underline}

\DeclareFontFamily{U}{mathb}{\hyphenchar\font45}
\DeclareFontShape{U}{mathb}{m}{n}{
      <5> <6> <7> <8> <9> <10> gen * mathb
      <10.95> mathb10 <12> <14.4> <17.28> <20.74> <24.88> mathb12
      }{}
\DeclareSymbolFont{mathb}{U}{mathb}{m}{n}
\DeclareMathSymbol{\boxasterisk}  {2}{mathb}{"66}
\providecommand{\convboxsharp}{\boxasterisk_{\sharp}}%

\makeatletter
\def\namedlabel#1#2{\begingroup
  \def\@currentlabel{#2}%
  \textbf{#2}
   \label{#1}\endgroup
}
\makeatother

 \newcommand{\kw}{motivic homotopy theory, coefficient system, six-functor formalism, exponential motives, varieties with potential, Landau-Ginzburg models}
\newcommand{\ttl}{Exponentiation of coefficient systems and exponential motives}

\title{\sffamily\sc \ttl}
\AtEndPreamble{\hypersetup{
    pdfcreator    = {\LaTeX{}},
    pdfencoding   = auto,
    psdextra,
    pdfauthor     = {Martin Gallauer and Simon Pepin Lehalleur},
    pdftitle      = {\ttl},
    pdfsubject    = {\ttl},
    pdfkeywords   = {\kw}}}
\begin{document}
\author{Martin Gallauer and Simon Pepin Lehalleur\footnotewomarker{\textit{Keywords:} \kw}}
\maketitle
\begin{abstract}
\noindent We construct new six-functor formalisms capturing cohomological invariants of varieties with potentials. Starting from any six-functor formalism~$C$, encoded as a coefficient system, we associate a new six-functor formalism~$\ep{C}$. This requires in particular constructing the convolution product symmetric monoidal structure at the \icategorical level. We study~$\ep{C}$ and how it relates to $C$. We also define motives in $\ep{C}$ attached to varieties with potential and study their properties.
\setcounter{tocdepth}{2}
\end{abstract}
\medskip
\tableofcontents

\section{Introduction}
\label{sec:intro}

\subsection*{Exponential cohomological invariants}
\label{sec:expon-cohom-invar}

A scheme with potential (or Landau-Ginzburg model) is a pair $(X,\pot)$ where $X$ is a scheme and $\pot\in \cO(X)$ is a regular function, which we think of as a morphism $\pot:X\to\aone\simeq \Ga{}$ to the additive group. Schemes with potential occur naturally in several areas of algebraic geometry: nearby and vanishing cycles, mirror symmetry for Fano varieties~\cite{mirror-fano}, Landau-Ginzburg/Calabi-Yau correspondence for quasihomogeneous potentials~\cite{fjr:LG-CY}, Donaldson-Thomas theory~\cite{ks-coha}, etc.

Schemes with potential have ``exponential'' cohomological invariants which generalize classical Weil cohomologies for ordinary schemes. For instance, if $X$ is smooth over a field~$k$ of characteristic $0$, one can consider the \emph{twisted de\,Rham cohomology} $\Hm^{*}_{\mathrm{dR}}(X,\cE^{\pot})$ where $\cE^{\pot}:=(\cO_{X},d-d\pot\wedge (-))$ is the exponential line bundle with connection associated to the function~$\pot$; this example gives its name to the general theory. If $X$ is locally of finite type over $\CC$, there is the \emph{rapid decay cohomology} $\Hm^{*}_{\mathrm{rd}}(X^{\an},\Z)$ which is roughly speaking the singular cohomology of a generic fiber of $\pot^{\an}:X^{\an}:=X(\CC)\to\CC$. Over a field $k$ of characteristic $p>0$, we have the \emph{twisted $\ell$-adic cohomology} $\Hm^{*}_{\ell}(X_{\bar{k}},\pot^{*}\lcl_{\psi})$ where $\lcl_{\psi}\in \mathrm{Sh}_{\et}(\Ga{\bar{k}},\overline{\Q}_{\ell})$ is the Artin--Schreier sheaf attached to an additive character $\psi:\FF_{p}\to\overline{\Q}_{\ell}^{\times}$.

When $\pot$ is the zero function, these cohomology theories recover the standard de\,Rham, Betti and $\ell$-adic cohomology groups of $X$, respectively. In general, however, they are genuinely new invariants which cannot be the realization of classical motives. For instance, twisted de\,Rham cohomology groups carry an \emph{irregular Hodge filtration} with rational breaks which does not fit into a Hodge structure~\cite{yu:hodge-filtration,esnault-sabbah-yu:irregular-hodge-degeneration}. And the comparison between twisted de\,Rham cohomology and rapid decay cohomology for varieties over number fields leads to \emph{exponential periods}, which form a $\Qbar$-subalgebra of $\CC$ containing~$e$ and the Euler-Mascheroni constant $\gamma$ and is (conjecturally) strictly larger than the $\Qbar$-algebra of classical periods~[\citealp[\S\,4.3]{kontsevich-zagier:periods}, \citealp{fresan-jossen:expmot}]. Finally, twisted $\ell$-adic cohomology over finite fields gives rise, via taking traces of Frobenius, to many interesting $\overline{\Q}_{\ell}$-exponential sums such as Kloosterman sums~\cite[Exposé~6]{deligne:sga4.5} \cite{fresan-sabbah-yu:hodge-kloosterman}.

Each of these cohomology theories is related to a sheaf theory: holonomic $\cD$-modules in the de\,Rham case, $\ell$-adic sheaves in the $\ell$-adic case and enhanced ind-sheaves (in the sense of~\cite[\S\,4]{dagnolo-kashiwara:riemann-hilbert}) in the Betti case. Exponential cohomology is obtained as the cohomology of an ``exponential local system'': $\cE^{\pot}$ in the de\,Rham case, $\pot^{*}\lcl_{\psi}$ in the $\ell$-adic case and the exponential enhanced (ind-)sheaves of~\cite[\S 3]{dagnolo-kashiwara:fourier-sato} in the Betti case\footnote{We could not find a precise statement of the comparison between the framework of d'Agnolo-Kashiwara and rapid decay cohomology, but apparently this comparison is known to experts.}. When $\pot$ is not constant, the exponential local system $\cE^{\pot}$ has \emph{irregular singularities} at infinity, and this is another reason why its cohomology is not of motivic origin in the classical sense, because of the regularity of Gauss--Manin connections.

The emerging theory of exponential motives aims to generalize results and conjectures about classical motives to varieties with potential. The idea of such a theory was first suggested by Kontsevich and Zagier in~\cite[\S 4.3]{kontsevich-zagier:periods}, developed in the context of \emph{exponential mixed Hodge structures} by Kontsevich and Soibelman in~\cite[\S 4]{ks-coha} and explored in the context of \emph{exponential Nori motives} by Fres\'{a}n and Jossen~\cite{fresan-jossen:expmot}. Fres\'{a}n and Jossen define, for $k\subset \CC$, a neutral $\Q$-linear Tannakian category $\mathbf{M}^{\exp}(k)$ with two fiber functors corresponding to twisted de\,Rham and rapid decay cohomology. The category $\mathbf{M}^{\exp}(k)$ contains as a full subcategory the abelian category of classical Nori motives~\cite[Theorem 5.1.1]{fresan-jossen:expmot}. This construction allows Fres\'{a}n and Jossen, among other things, to define a period torsor and to state and study a generalization of the Grothendieck-Kontsevich-Zagier period conjecture for exponential periods~\cite[Conjecture 1.3.2]{fresan-jossen:expmot}.

\subsection*{Outline}
\label{sec:outline--content}

In this paper, we define categories of exponential motivic sheaves in the context of motivic homotopy theory after Morel and Voevodsky. Our approach is centered on the six-functor formalism, which we approach through the notion of \emph{coefficient systems}. A coefficient system is an axiomatization of the common properties satisfied both by classical sheaf theories and by the stable motivic homotopy theory~$\SH$ (see \cref{defn:cosy}). As we recall in \cref{sec:cs}, a coefficient system can be extended to a six-functor formalism satisfying the usual properties.

We work over a base scheme~$\base$. Given a coefficient system~$C$ on $\sch[\base]$, we construct an associated \emph{exponential coefficient system}~$\ep{C}$ (\Cref{sta:Cexp-cosy}).  The definition, inspired by~\cite{ks-coha}, is very simple. For a $\base$-scheme~$X$, writing $\pi:\Ga{X}\to X$ for the structure morphism of the additive group over~$X$, we set
\[
\ep{C}(X)\ :=\ \{\, M\in C(\Ga{X})\ \mid\ \pi_{!}M=0\,\}
\]
(or equivalently the Verdier quotient of~$C(\Ga{X})$ by the constant objects~$\pi^{*}C(X)$, see \cref{sec:exp-subcats}) which we equip with the convolution product
\[
M\conv N\ :=\ \plus_{!}\left(M\boxtimes_{X} N\right)
\]
where $\plus:\Ga{X}\times\Ga{X}\to \Ga{X}$ is the addition map. The inclusion functor $\ep{C}(X)\to C(\Ga{X})$ admits a right adjoint
\[
\Pi:C(\Ga{X})\to \ep{C}(X),\qquad M\mapsto M\conv \E{0},
\]
where $\E{0}\in\ep{C}(X)$ is the unit for the convolution product.\footnote{All of this actually works for some groups other than~$\Ga{}$ as well; see \Cref{conv:group-scheme}.}

Once we establish that~$\ep{C}$ forms a coefficient system, the general theory provides us with a six-functor formalism. In particular, for a morphism $f:Y\to X$, we have adjunctions
\[
\ul{f}^{*}:\ep{C}(X)\rightleftarrows \ep{C}(Y):\ul{f}_{*}\quad\text{and}\quad
\ul{f}_{!}:\ep{C}(Y)\rightleftarrows \ep{C}(X):\ul{f}^{!}
\]
under (the familiar) mild assumptions on~$f$.
(Here, we underline these functors to distinguish them from those in~$C$.). We describe these functors as well as the internal Homs of~$\ep{C}$ explicitly in terms of the operations in~$C$ and the functor~$\Pi$ (\cref{sta:identify-four-functors-hexp,sta:int-homs}).
Consequently, the exponential theory~$\ep{C}$ remains as computable as~$C$ itself.

We can then apply this general construction. From the coefficient system~$\SH$ of \emph{stable motivic homotopy types} we obtain a six-functor formalism~$\ep{\SH}$ of \emph{exponential stable motivic homotopy types}. Similarly, from the coefficient system~$\DM$ of \emph{motivic sheaves}, we obtain a six-functor formalism~$\ep{\DM}$ of \emph{exponential motivic sheaves}.
We refer the reader to \cref{sec:cs-examples} for other examples.

The simple-looking definition above hides a difficulty. We want our six-functor formalisms to be valued in symmetric monoidal stable \icats rather than bare tensor triangulated categories. Such an enhancement exists in most examples of interest and is extremely useful to treat questions of descent and equivariance, for example. The difficulty, then, is to make sense of the convolution product on~$\ep{C}(X)$ as a symmetric monoidal \icat. We propose two approaches in this paper (\cref{sec:cs-icats}). The first one, which we only develop informally, is to use the idea, originally due to Lurie, of encoding six-functor formalisms with correspondences. This idea has been developed in~\cite{gaitsgory-rozenblyum:vol1}, with an extension to the lax-monoidal setting in~\cite[Appendix~A]{richarz-scholbach:motivic-satake}. This conceptually satisfying approach, however, involves $(\infty,2)$-categories whose foundations aren't as rigorously established. For this reason, we give another definition of the convolution product via a more cumbersome but strictly $(\infty,1)$-categorical construction, relying on symmetric monoidal coCartesian fibrations (\cref{sec:cs-icats}).\footnote{Note that recent work of Mann \cite[Appendix~A]{mann:padic} could provide a new approach to six-functor formalisms via correspondences in our setting without $(\infty,2)$-categories and make the first construction of convolution unconditional.}

The definition of~$\ep{C}(\base)$ does not involve at face value any $\base$-schemes with potential. If $C=\SH$ is Morel-Voevodsky's stable motivic homotopy theory (the initial coefficient system by~\cite{drew-gallauer:usf}), then $\ep{\SH}(\base)$ is a subcategory of~$\SH(\Ga{\base})$: By definition of $\SH$, this means that $\ep{\SH}(\base)$ is built out of smooth $\Ga{\base}$-schemes, or in other words, $\base$-schemes $f:X\to \base$ with a \emph{smooth} potential $\pot:X\to \Ga{}$. However, with the six-functor formalism at hand, we can actually construct exponential motives for arbitrary $\base$-schemes with potential. We first construct in \cref{sec:exponential-twists}, for any coefficient system $C$ and any potential $\pot:X\to \Ga{}$, an \emph{exponential twist} $\E{\pot}\in \ep{C}(X)$ which plays the role of the exponential $\cD$-module~$\cE^{\pot}$. As mentioned above, the object~$\E{0}$ associated to the zero potential is the unit for the convolution product. We then define the corresponding \emph{exponential (homological) motive}~$M_{\base}(X,\pot)\in\ep{C}(\base)$ as
\[
M_{\base}(X,a)\ :=\ \underline{f}_{!}\left(\E{-\pot}\conv\underline{f}^{!}\E{0}\right)
\]
and study basic properties of this construction in \cref{sec:motives}. This procedure allows us, under some assumptions, to construct generators of the categories~$\ep{C}(\base)$ (\cref{sec:generators}). In particular, we study generators of the categories~$\ep{\DM}(\base)$ and~$\ep{\SH}(\base)$. This also leads us to define \emph{exponential (motivic) cohomology groups} as certain Hom groups in~$\ep{C}(X)$, which we compute in terms of morphism groups in the original coefficient system~$C$ (\cref{prop:exp-mot-coh-comp}).

\subsection*{Related and Future work}\label{sec:rel-future}

With this basic formalism in place, we plan to explore other constructions specific to exponential motives. One of them is the Fourier transform. Sheaf-theoretic Fourier transforms exist in various contexts: Fourier-Laplace transform of holonomic $\cD$-modules in characteristic~$0$, Fourier-Deligne transform for $\ell$-adic sheaves in positive characteristic~\cite{Laumon-epsilon}, enhanced Fourier-Sato transform for enhanced ind-sheaves~\cite{dagnolo-kashiwara:fourier-sato}, etc. On the other hand, it is well-known that no motivic lift of these operations exists classically, so one motivation for exponential motives is that they admit such a Fourier transform, satisfying all the standard properties of the Fourier-Deligne transform. At the triangulated level, it is simply defined, for $M\in C_{\exp}(\Ga{})$, as:
\[
\FT(M)\ :=\ \ul{p_{1}}_{!}\left(\ul{p_{2}}^{*}M\conv \ul{\plus}^{*}\E{\id}\right)[+1]
\]
Following ideas of Beraldo~\cite[\S 5.1]{beraldo:loop-groups-acts} involving correspondences decorated by a potential, we will construct an \icategorical enhancement of this functor.

We will construct realization functors which connect~$\ep{\SH}$ and~$\ep{\DM}$ to classical sheaf theories and recover the exponential cohomology theories (rapid decay, exponential de\,Rham, twisted $\ell$-adic) that motivated the theory in the first place. For instance, if $X$ is a finite type scheme over a field~$k$ of positive characteristic~$p$ and~$\psi$ an additive character of $\FF_{p}$ as above, we will define a twisted $\ell$-adic symmetric monoidal realization functor (depending on~$\psi$) as
\[
\mathrm{R}_{\exp,\ell}^{\psi}:\ep{\DM}(X,\Q)\to \mathrm{D}(X_{\et},\Qellbar),\qquad M\mapsto i_{1}^{*}\FT_{\ell,\psi}(\mathrm{R}_{\ell}(M))[+1],
\]
where the functor $\mathrm{R}_{\ell}$ is the $\ell$-adic realization for classical motives on $\Ga{X}$, the functor $\FT_{\ell,\psi}$ is the $\ell$-adic Fourier--Deligne transform for $\psi$ and $i_{1}:\Spec(k)\to\Ga{k}$. This definition may seem complicated but it is a convenient way to construct $\mathrm{R}_{\exp,\ell}^{\psi}$ as a symmetric monoidal functor, because of how the Fourier transform interacts with convolution products~\cite[Proposition~1.2.2.7]{Laumon-epsilon}.

We will also connect our construction to the category~$\mathbf{M}^{\exp}(k)$ of exponential Nori motives~\cite{fresan-jossen:expmot} (where $k$ is a subfield of~$\CC$) by constructing a symmetric monoidal realization functor
\[
R_{\exp,\mathrm{Nori}}:\DM_{\mathrm{exp},\textup{c}}(k,\Q)\to \mathrm{D^{b}}(\mathbf{M}^{\exp}(k))
\]
which in the classical setting is due to Nori~\cite{nori-lectures,harrer:phd}.
We expect that $R_{\exp,\mathrm{Nori}}$ can be used to establish an isomorphism between the exponential motivic Galois group~$G^{\exp}(k)$ of~\cite{fresan-jossen:expmot} and a truncation of the derived Tannakian dual of~$\ep{\DM}(k,\Q)$ for the Betti realization, just as in the classical setting~\cite{choudhury-gallauer:ayoub-nori}.

As we point out in \cref{rmk:tate-mot}, the category of \emph{exponential mixed Tate motives} $\ep{\DTM}(X,\Q)$ on a scheme $X$ satisfying the Beilinson-Soulé conjecture admits a t-structure by \cite{levine:tate}. We plan to study this category and in particular the associated exponential periods when $X$ is the spectrum of a number field.

These and similar advances will rely on a good understanding of the exponential motivic cohomology whose study we initiate here.
One property that we presently do not discuss is Poincar\'e duality.
That property seems to be much more subtle in the exponential setting and we will come back to it in future work.

The idea of adding a coordinate to encode ``exponential'' phenomena and irregular singularities has been used in several guises in recent advances on the irregular Riemann-Hilbert correspondence and irregular Hodge theory, through the aforementioned enhanced sheaves of Tamarkin~\cite{tamarkin:microlocal} and d'Agnolo-Kashiwara~\cite{dagnolo-kashiwara:riemann-hilbert} (using sheaves on $X\times \RR$ where $X$ is a real or complex manifold), the mixed twistor $\cD$-modules of Mochizuki~\cite{mochizuki:mtm} and the irregular mixed Hodge modules of Sabbah-Yu~\cite{sabbah:irregular-hodge} (using sheaves on $X\times \CC$).

Another approach to exponential motivic homotopy theory, which was suggested by Marco Robalo, would be to take schemes with potential as the primary object of study and to directly construct an exponential stable motivic homotopy category ``à la Morel--Voevodsky''. More precisely, starting with Nisnevich sheaves of spaces on the slice category $\sm{X}/\Ga{S}$ of $X$-schemes with potential,  imposing $\aone$-homotopy invariance, and $\otimes$-inverting suitable objects, one obtains a theory that may be closely related to our $\ep{SH}(-)$.
We intend to explore this relationship in the future. A third approach to exponential motives, based on foliations, is currently being studied by Hennion-Holstein-Robalo in the context of gluing categories of matrix factorizations in Donaldson-Thomas theory \cite{robalo:slides}.

\subsection*{Acknowledgments}
\label{sec:acknowledgments}
We thank Javier Fres\'{a}n for many helpful discussions and his constant support, and Marco Robalo for informing us about his approach to the subject.

MG was supported by a Titchmarsh Research Fellowship of the University of Oxford. He also thanks the Max Planck Institute and the Hausdorff Research Institute for Mathematics in Bonn for their hospitality. SPL was supported by the DFG via the SPP 1786 and by the Netherlands Organisation for Scientific Research (NWO), under the TOP grant of Ben Moonen and Lenny Taelman (project number 613.001.752). The SPP 1786 (resp. the TOP grant) also financed a research visit of MG and Javier Fres\'{a}n to Berlin (resp. of MG to Nijmegen). 
\section{Conventions}
\label{sec:conventions}

\subsection{Algebraic geometry}

Throughout we fix a Noetherian scheme~$\base$ which serves as our base scheme. All our schemes are of finite type over~$\base$, and we denote the category of $\base$-schemes of finite type by $\sch[\base]$.
\subsection{Category theory}

Generally, we follow the notations of~\cite{HTT,HA} for \icats and symmetric monoidal \icats.
Ordinary $1$-categories are identified with their nerves viewed as \icats.
We denote symmetric monoidal \icats by symbols like $C^{\otimes}$ when we want to emphasize the monoidal structure~$C^\otimes\to\Fin_*$, but we sometimes identify~$C^\otimes$ with its underlying \icat~$C$ when the monoidal structure is clear from the context.
Given objects~$X,Y$ in a stable \icat~$C$, their mapping spectrum is denoted by $\mapsp(X,Y)\in\Sp$.

We denote by $\Cat{st}$ the \icat of stable \icats and exact functors. It admits a symmetric monoidal structure~\cite[5.1.1]{hermitian1} for which functors $C\otimes D\to E$ are the same as biexact functors $C\times D\to E$.
The unit is given by the \icat of finite spectra $\Sp^{\omega}$.
We thereby identify stably symmetric monoidal \icats with $\calg{\Cat{st}}$.
A similar discussion applies to the \icat $\Prst$ (resp. $\Prstomega$) of stable presentable (resp. stable compactly generated) categories, see e.g. \cite[\S 4.8.2]{HA} for the stable presentable case.

Given a symmetric monoidal $\infty$-category $C^{\otimes}$, we denote by $\Mod{C}$ the generalised \ioperad of module objects in $C^{\otimes}$. Informally, the objects of $\Mod{C}$ are pairs $(A,M)$ where $A$ is a commutative algebra object in $C^{\otimes}$ and $M$ is an $A$-module, and morphisms from $(A,M)$ to $(A',M')$ are pairs $(\phi:A\to A',f:M\to M')$ where $\phi$ is a morphism of commutative algebras and $f$ is a morphism of $A$-modules, with $M'$ considered as an $A$-module via $\phi$. In particular, we have $\Mod{\Cat{st}}$, which admits a forgetful functor (a map of generalised \ioperads) to $\calg{\Cat{st}}$.

Given two \icats $C$ and $D$, we denote by $\Fun(C,D)$ the corresponding functor \icat.
If $D$ is the underlying \icat of a symmetric monoidal \icat $D^{\otimes}$, then the functor \icat $\Fun(C,D)$ has a corresponding symmetric monoidal structure $\Fun(C,D)^{\otimes}$ given by the pointwise symmetric monoidal structure (this is the special case of~\cite[Remark 2.1.3.4]{HA} with $\mathcal{O}^{\otimes}=\Fin_*$ the commutative \ioperad).

\section{Coefficient systems}
\label{sec:cs}

The notion of coefficient system provides a convenient framework for discussing theories of (derived) sheaves on schemes in algebraic geometry, with a view towards the formalism of Grothendieck's six functors.
It is due to Voevodsky and Ayoub~\cite{deligne:cross-functors,ayoub:thesis-1} (in a 2-categorical setting) and further studied in Cisinski-Déglise~\cite{cisinski-deglise:trig-cats}, Drew~\cite{drew:motivic-hodge-modules} (in our current infinity-categorical setting) (see also~\cite{khan-phd,robalo:thesis, hoyois:equivariant,ayoub:anabel}).
For a leisurely introduction we refer to~\cite{gallauer:six-functor-survey}.
The term ``coefficient system'', suggested by~\cite{drew:motivic-hodge-modules} after Grothendieck~\cite{grothendieck:rets}, is meant to suggest that we are parametrizing the possible coefficients of various sheaf cohomology theories.

As we discuss in \Cref{sec:cs-examples}, classical sheaf theories (sheaves for the analytic topology, $\ell$-adic sheaves, $\mathcal{D}$-modules) provide important examples of coefficient systems. The definition is nevertheless particularly suitable for theories coming from motivic homotopy theory, since it axiomatises the operations and properties which are formally available from the Morel-Voevodsky construction.

\subsection{Definitions}

\begin{defn}
\label{defn:cosy}%
A \emph{coefficient system (over $\base$)} is a functor
$C:\schop[\base]\to\calg{\Cat{st}}$ satisfying the following axioms (where we write $f^{*}=C(f)$ for a morphism $f$ of $\base$-schemes).
\begin{enumerate}[(1)]
\item
\namedlabel{ax:left}{(Left)}
\label{Ax:left}%
For each smooth morphism $p:Y\to X\in\sch[\base]{}$, the functor $p^*:C(X)\to C(Y)$ admits a left adjoint $p_\sharp$, and we have the following equivalences.
\begin{description}
\item[\namedlabel{ax:smooth-bc}{(Smooth base change)}]
For each cartesian square
\[
\begin{tikzcd}
Y'
\ar[r, "p'" above]
\ar[d, "f'" left]
&
X'
\ar[d, "f" right]
\\
Y
\ar[r, "p" above]
&
X
\end{tikzcd}
\]
in $\sch[\base]{}$, the Beck-Chevalley transformation $p'_\sharp (f')^*\to f^*p_\sharp$ is an equivalence.
\item[\namedlabel{ax:projection}{(Smooth projection formula)}] The canonical transformation
\[
p_\sharp(p^*(-)\otimes -)\to -\otimes p_\sharp(-)
\]
is an equivalence.
\end{description}
\item
\namedlabel{ax:right}{(Right)}
For every $X\in\sch[\base]{}$ and every $f:Y\to X$, the following properties hold.
\label{Ax:right}%
\begin{description}
\item[\namedlabel{ax:closed}{(Internal hom)}]
The symmetric monoidal structure on $C(X)$ is closed.
\item[\namedlabel{ax:pushforwards}{(Push-forward)}]
The pull-back functor $f^*$ admits a right adjoint $f_*:C(Y)\to C(X)$.
\end{description}

\item
\namedlabel{ax:loc}{(Localization)}
The \icat $C(\emptyset)\simeq 0$ is trivial.
And for each closed immersion $i:Z\into X$ in $\sch[\base]{}$ with complementary open immersion $j:U\into X$, the square
\begin{equation}
\label{eq:def-cosy-loc}
\begin{tikzcd}
C(Z)
\ar[r, "{i_*}"]
\ar[d]
&
C(X)
\ar[d, "j^*"]
\\
C(\emptyset)\simeq 0
\ar[r]
&
C(U)
\end{tikzcd}
\end{equation}
is Cartesian in $\Cat{st}$.
\item
\label{Ax:hty-stable}%
For each $X\in\sch[\base]{}$, if $p:\affine^1_X\to X$ denotes the canonical projection with zero section $s:X\to\affine^1_X$, then:
\begin{description}
\item[\namedlabel{ax:hty}{($\affine^1$-homotopy)}]
The functor $p^*:C(X)\to C(\affine^1_X)$ is fully faithful.
\item[\namedlabel{ax:stable}{(Tate stability)}]
The composite $p_{\sharp} s_*:C(X)\to C(X)$ is an equivalence.
\end{description}

\end{enumerate}
\end{defn}

\begin{rmk}\label{rmk:cs-non-mon}
It is also sometimes convenient to work with the variant of this definition without symmetric monoidal structures, i.e. a functor  $C:\schop[\base]\to\Cat{st}$ satisfying the same axioms except the smooth projection formula and the internal homs. 
\end{rmk}

\begin{rmk}
\label{rmk:smooth-projection-formula}%
Let us expand a little on condition \ref{ax:projection}. The functor~$p^{*}$, like any symmetric monoidal functor in $\calg{\mCat{st}}$, has a canonical upgrade to a morphism $(C(Y),C(Y))\to(C(Y),C(X))$ in $\Mod[]{\Cat{st}}$, where $C(Y)$ acts on $C(X)$ via $p^{*}$ itself. In particular, if we write $m:C(Y)\times C(X)\to C(X)$ for the action map, we have a commutative square
  \[
  \begin{tikzcd}
C(Y)\times C(Y)
\ar[r,"\id\times p^{*}"]
\ar[d,"\otimes_{C(Y)}"]
&
C(Y)\times C(X)
\ar[d,"m"]
\\
C(Y)
\ar[r,"p^{*}"]
&
C(X)
\end{tikzcd}
\]
and the canonical map in \ref{ax:projection} is the mate of this square when passing to left adjoints. The condition \ref{ax:projection} is that this mate is an isomorphism. By~\cite[Proposition 9.4.15]{robalo:thesis}, this is equivalent to $p_{\sharp}$ lifting to a map $(C(Y),C(X))\to (C(Y),C(Y))$ in $\Mod{\Cat{st}}$. (More precisely, \cite[Proposition 9.4.15]{robalo:thesis} works in~$\Mod{\Prst}$, but the proof works in the non-presentable setting as well.)
\end{rmk}  

\begin{defn}
  A coefficient system $C$ over $\base$ is \emph{presentable} if it factors through the subcategory $\calg{\Prst}$ of stable presentably symmetric monoidal \icats.
  In other words, for every $X\in \sch[\base]$, the symmetric monoidal \icat $C(X)$ is presentably symmetric monoidal, and for every morphism $f:Y\to X$, the functor $f^{*}:C(X)\to C(Y)$ is cocontinuous (or equivalently is a left adjoint).

Similarly, we say that $C$ is \emph{compactly generated} if it factors through the subcategory $\calg{\Prstomega}$ of compactly generated symmetric monoidal \icats and compact-preserving cocontinuous functors. 
\end{defn}

\begin{rmk}
For a functor $C:\sch[\base]\op\to\calg{\Prst}$ taking values in stable presentable \icats, the adjoint functor theorem \cite[Corollary 5.5.2.9]{HTT} allows to reformulate part of the axioms.
\begin{enumerate}
\item The existence of $f_{*}$ is automatic
(since $f^{*}$ is cocontinuous), and so is the existence of internal homs (since the symmetric monoidal product is cocontinuous in each variable).
In other words, the axiom \ref{ax:right} is automatic.
\item  The existence of $p_{\sharp}$ for
$p$ smooth is equivalent to $p^{*}$ being continuous and accessible.
And the axiom \ref{ax:projection} can also be expressed as $p_{\sharp}$ lifting to a
map in $\Mod{\Prst}$.
\end{enumerate}

\end{rmk}

\begin{defn}
\label{defn:cs-morphism}%
  \begin{enumerate}[label=(\roman{*})]
  \item   Let $C_{1},C_{2}$ be coefficient systems over $\base$. A \emph{morphism of coefficient systems} is a morphism $\phi:C_{1}\to C_{2}$ in the functor $\infty$-category $\Fun(\schop[\base],\calg{\Cat{st}})$ with the following property: for every morphism $f:Y\to X$ in $\sch[\base]$, the Beck-Chevalley transformation
    \[
f_{\sharp}\phi_{Y}\to \phi_{X}f_{\sharp}:C_{1}(Y)\to C_{2}(X)
\]
is invertible.
\item The $\infty$-category $\CoSy{\base}$ of coefficient systems over $\base$ is defined as the  subcategory of $\Fun(\schop[\base],\calg{\Cat{st}})$ whose objects are coefficient systems and $1$-morphisms are defined as above.
We also write $\COSY{\base}$ (resp. $\cCOSY{\base}$) for the subcategory of $\CoSy{\base}$ spanned by presentable \resp{compactly generated} coefficient systems and cocontinuous \resp{and compact-preserving} morphisms.
  \end{enumerate}
\end{defn}

In practice, \ref{ax:loc} is often the most difficult to verify.
In order to establish that it is preserved by certain constructions, we will find it useful to reformulate the axiom as follows.

\begin{lem}
\label{sta:localization-axiom}%
Let $C\in\Fun(\schop[\base],\calg{\Cat{st}})$ be a functor satisfying \ref{ax:pushforwards} and \ref{ax:smooth-bc}.
Then the following are equivalent:
\begin{enumerate}[(i)]
\item
\label{it:localization-axiom.1}%
$C$ satisfies \ref{ax:loc}.
\item
\label{it:localization-axiom.2}%
For each closed immersion $i$ with complementary open immersion $j$,
the functor $C$ satisfies all of:
\begin{enumerate}[(1)]
\item $C(\emptyset)\simeq 0$;
\item $i_*$ is fully faithful;
\item the pair $(i^*,j^*)$ is jointly conservative.
\end{enumerate}
\end{enumerate}
\end{lem}
\begin{proof}
This is well-known, see for example~\cite[Remark~5.9.(3)]{drew:motivic-hodge-modules} or~\cite[\S\,I.2.3]{cisinski-deglise:trig-cats}.
\end{proof}

\begin{rmk}
\label{rmk:localization}%
If $C$ is a coefficient system and $i,j$ as in \Cref{sta:localization-axiom} then one automatically has the following additional properties:
  \begin{enumerate}[(1)]
  \item\label{it:localization.i-upper-shriek} $i_*$ admits a right adjoint $i^!$.
  \item $j^*i_*\simeq 0$.
  \item $j_*, j_\sharp$ are fully faithful.
  \item There are cofiber sequences for all $M$:
    \[
      j_\sharp j^*M\to M\to i_*i^*M,\qquad i_*i^!M\to M\to j_*j^*M
    \]
  \end{enumerate}
The functors $i^{*}$ and $j^{*}$ thus exhibit $C(X)$ as a \emph{recollement} of $C(Z)$ and $C(U)$ in the sense of~\cite[Appendix~A.8]{HA}
\end{rmk}

\subsection{Examples}
\label{sec:cs-examples}

  Many existing ``sheaf theories'' form coefficient systems. All those examples of coefficient systems are either small or presentable, with small coefficient systems occuring as subcategories of compact or constructible objects in presentable coefficient systems. For some of these examples, the necessary constructions and the verification of the axioms of coefficient systems are not available in a convenient form in the literature. See \cite[\S 6]{abe:bivariant} for another discussion of examples.

\subsubsection{Motivic homotopy theory} Stable motivic homotopy theory $\SH(-)$, in its \icategorical incarnation constructed in~\cite[\S 9.1]{robalo:thesis} forms a presentable coefficient system over any base~$\base$~\cite[Theorem 9.4.36]{robalo:thesis}. Even though the definition of coefficient system does not directly refer to motivic homotopy theory, the coefficient system $\SH(-)$ plays a distinguished role because it is the initial object in $\COSY{\base}$ by~\cite[Theorem~7.14, Remark~7.15]{drew-gallauer:usf}.

  The arguments of~\cite{robalo:thesis} also show that $\DA^{\Nis}(-)$ (the $\aone$-derived category of Morel), $\SH^{\et}(-)$ (the étale local version of stable motivic homotopy theory) and $\DA^{\et}(-)$ (the étale motives of Morel) form coefficient systems over any base~$\base$. 
\subsubsection{Motives}\label{ex:spitzweck} Spitzweck~\cite{MR3865569} constructs a motivic ring spectrum $\cal{M}\in\SH(\Spec(\Z))$ that represents Bloch-Levine motivic cohomology and then defines
\[
\DM(-):=\SH(-;\cal{M}):\schop\to\calg{\Prst}
\]
as the functor which associates with the scheme~$X$ the \icat of modules in $\SH(X)$ over the pullback of $\cal{M}$ to~$X$.
See~\cite[Theorem~8.10]{drew:motivic-hodge-modules} for more details on the general construction of new coefficient systems as modules. 

Spitzweck proves that over a field~$k$ its compact part is equivalent to Voevodsky's category of geometric motives~\cite{voevodsky00-mm},
while with rational coefficients (and for any scheme) one recovers Beilinson motives in the sense of~\cite{cisinski-deglise:trig-cats} (which are themselves equivalent to $\DA^{\et}(-,\Q)$ by~\cite[Theorem 16.2.18]{cisinski-deglise:trig-cats}). Consequently, $\DM(-)$ can be seen as a coefficient system of integral motivic sheaves.

\subsubsection{Analytic sheaves} \label{ex:cs-analytic} Let $k$ be a field of characteristic $0$ equipped with a complex embedding $\sigma:k\to \CC$ and let $\Lambda\in\calg{\Sp}$ be a commutative ring spectrum. There is a presentable coefficient system $\Shv((-)^{\an},\Lambda)$ over the base $\Spec(k)$, whose objects are sheaves of $\Lambda$-modules for the analytic topology (see for instance~\cite[Proposition 1.26]{ayoub:anabel}, where coefficient systems are called ``Voevodsky pullback formalisms'').
When $\Lambda$ is a discrete ring, this is equivalent to the unbounded derived category of sheaves of $\Lambda$-modules. The subfunctor of algebraically constructible sheaves forms a small coefficient system $\Shv_{\ct}((-)^{\an},\Lambda)$, and one can take the Ind-completion~\cite[\S\,2.4.2]{gallauer:six-functor-survey} to get a compactly generated coefficient system of ind-constructible sheaves $\Ind(\Shv_{\ct}((-)^{\an},\Lambda))$ (this is~\cite[Corollary 1.27]{ayoub:anabel}, with slightly different notations). The Betti realization functor can then be interpreted as the morphism of coefficient systems  $\SH(-)\to \Shv((-)^{\an},\Lambda)$ or $\SH(-)\to\Ind(\Shv_{\ct}((-)^{\an},\Lambda))$~\cite[Theorem 1.28]{ayoub:anabel}.
\subsubsection{Nori motivic sheaves}  Let $k$ be as in \cref{ex:cs-analytic}. In~\cite[\S 2.3]{ayoub:anabel}, Ayoub defines a presentable coefficient system  $\Shv_{\mathrm{geo}}^{\mathcal{G}_{\mot}^{\mathrm{cl}}}(-,\Q)$ of ``(triangulated) Nori motivic sheaves''. For $X\in \sch[k]$, the stable $\infty$-category $\Shv_{\mathrm{geo}}^{\mathcal{G}_{\mot}^{\mathrm{cl}}}(X,\Q)$ is equipped with a t-structure whose heart is equivalent to the indization of the abelian category of finite-dimensional representations of Ayoub's motivic Galois group $\mathcal{G}_{\mot}^{\mathrm{cl}}(k,\Q)$ of $k$, or equivalently by~\cite{choudhury-gallauer:ayoub-nori}, to the indization of the abelian category of Nori motives over $k$. This construction thus provides a reasonable candidate for ``Nori motivic sheaves'' with an associated six-functor formalism. Ayoub suggests as a precise relative version of the motivic t-structure conjecture that the natural morphism of coefficient systems $\DA^{\et}(-,\Q)\to \Shv_{\mathrm{geo}}^{\mathcal{G}_{\mot}^{\mathrm{cl}}}(-,\Q)$ induces an equivalence on constructible objects.
\subsubsection{\'{E}tale and \texorpdfstring{$\ell$}{l}-adic sheaves} \label{ex:cs-l-adic}A general construction of an \icat $\mathrm{D_{cons}}(X;\Lambda)$ of constructible \'etale sheaves on a scheme~$X$ with coefficients in a condensed ring~$\Lambda$ is performed in~\cite{hemo-richarz-scholbach:constr-sheaves}.
When applied to $\Lambda\in\{\Z_{\ell},\Q_{\ell},\overline{\Q}_{\ell}\}$, one recovers the classical theory of $\ell$-adic constructible sheaves.
Thus for example, the functor $\mathrm{D_{cons}}(-;\Q_\ell):\schop[k]\to\calg{\Cat{st}}$ defined on $k$-schemes for an algebraically closed field~$k$ is a coefficient system, by \Cref{sta:cosy-vs-triangulated} together with the classical results on $\ell$-adic sheaves in~\cite{deligne:sga4.5}. 
\subsubsection{Holonomic \texorpdfstring{$\cD$}{D}-modules} There should be a coefficient system $\mathrm{D_h^b}(\cD_{-})$ of holonomic $\cD$-modules on schemes over a field~$k$ of characteristic zero although we don't have a convenient reference for this claim.

\subsubsection{Mixed Hodge modules} Similarly, there should be an \icategorical lift of Saito's bounded derived category of mixed Hodge modules~\cite{MR1047415} although we do not know of a reference.
An alternative by Drew~\cite{drew:motivic-hodge-modules} provides a coefficient system of \emph{motivic} Hodge modules $\mathrm{DH}(-)$ on $\Spec(\CC)$.
It comes with well-behaved realization functors, shares many desirable properties with Saito's theory, and is conjectured to embed fully faithfully into the latter.
\subsubsection{Non-examples} There are important examples of ``sheaf theories'' with a rich functoriality but which do not form coefficient systems. Given a base scheme $S$, we have functors
\[
\mathrm{QCoh}(-),\ \mathrm{IndCoh}(-),\ \mathrm{Dmod}(-):\schop[S]\to\calg{\Cat{st}}
\]
(of quasi-coherent sheaves, ind-coherent sheaves, and general left D-modules) constructed in~\cite{GR1,GR2}. None of them are coefficient systems, for several reasons. For $\mathrm{QCoh}(-)$ and $\mathrm{IndCoh}(-)$, equipped with the natural pullback operation (which makes sense for morphisms of schemes), one issue is that $j^{*}$ does not have a left adjoint when $j$ is an open immersion, contradicting \ref{Ax:left}. For $\mathrm{Dmod}(-)$, the situation is even worse as only three out of six functors (cf.\ \Cref{sec:cs-properties}) are defined.
The work of Clausen-Scholze using condensed mathematics~\cite{clausen-scholze:condensed} leads to a six-functor formalism for a variant of $\mathrm{QCoh}(-)$ (``solid quasicoherent sheaves'') but even then \cref{Ax:hty-stable} fails and we do not get a coefficient system. 

\subsection{Six-functor formalism}
\label{sec:cs-properties}

In~\cite{ayoub:thesis-1} it is shown that the triangulated version of the axioms of coefficient systems affords the formalism of the six functors on quasi-projective schemes, and with slightly stronger assumptions, \cite{cisinski-deglise:trig-cats} extends this to all (separated finite type) schemes.
Most of this formalism has now been lifted to stable \icats, see for example~\cite{liu-zheng:6ff-artin-stacks,robalo:thesis,khan-phd,drew:motivic-hodge-modules,AGV}.
We recall some of the properties we will need in the sequel.
A more complete survey can be found in~\cite{gallauer:six-functor-survey}.
Fix a coefficient system $C\in\CoSy{\base}$.

\subsubsection{Exceptional functoriality} \footnote{The existence of the exceptional functors is typically proved for presentable coefficient systems and can be deduced for those taking values in small \icats by passing to Ind-objects~\cite[\S\,2.4.2]{gallauer:six-functor-survey}. A given coefficient system can always be viewed as taking values in small \icats in a suitable universe (if one's set-theoretic tenets allow for such) so that the following discussion applies to every coefficient system.}
An important fact is that there are exceptional functors associated with any morphism $f:Y\to X$ of $\base$-schemes:
\begin{align*}
  f_!:C(Y)\to C(X),&&f^!:C(X)\to C(Y),
\end{align*}
the first being left adjoint to the second.
These can be promoted to functors $C_!:\sch[\base]\to\Cat{st}$ and $C^!:\schop[\base]\to\Cat{st}$.
The latter is moreover a sheaf for the cdh-topology (as is $C=C^*$ itself).
When $f$ is proper, there is an equivalence $f_!\simeq f_*$, and when $f$ is an open immersion, there is an equivalence $f_!\simeq f_{\sharp}$.

\subsubsection{Linearity}
\label{cs-linearity}
If $f$ is a $\base$-morphism then the functor $f_!$ is `linear'/satisfies the projection formula in the sense that the canonical morphism
\[
f_!(M\otimes f^*N)\isoto f_!M\otimes N
\]
is an equivalence for any $M\in C(Y)$, $N\in C(X)$.
And equally, for a Cartesian square in $\sch[\base]$
\begin{equation}
\label{eq:linearity-square}
\begin{tikzcd}
Y'
\ar[r, "f'" above]
\ar[d, "g'" left]
&
X'
\ar[d, "g" right]
\\
Y
\ar[r, "f" above]
&
X
\end{tikzcd}
\end{equation}
the canonical morphism
\[
f'_!(g')^*\isoto g^*f_!
\]
is an equivalence.
By adjunction, there is also an equivalence
\[
f^!g_*\isoto g'_*(f')^!.
\]
\subsubsection{Smooth and proper base change}
\label{cs-basechange}

An important consequence of the previous two properties is \emph{proper base change}:
If in the Cartesian square above $f$ (and therefore $f'$) is proper then the canonical morphism
\[
g^*f_*\isoto f'_*(g')^*
\]
is an equivalence.
The same conclusion holds if instead $g$ (and therefore $g'$) is smooth.
This is essentially equivalent to the axiom \ref{ax:smooth-bc}.

\subsubsection{Relative purity}
\label{cs-relative-purity}
Given a locally free $\cO_X$-module $\mathcal{V}$ viewed as a vector bundle $p:V\to X$ with zero section $s:X\into V$, one denotes by
\[
\thom(\mathcal{V})=\thom(V):=p_\sharp s_*:C(X)\isoto C(X)
\]
the associated Thom equivalence.
Thom equivalences behave well under base change in~$X$ and direct sums of vector bundles.
The inverse is denoted by $\thom^{-1}(\mathcal{V})=s^!p^*$. We have in fact a canonical isomorphism $\thom(\mathcal{V})(M)\simeq \thom(\mathcal{V})(\one)\otimes M$, and we denote by
\[
\mthom{\mathcal{V}}:=\thom(\mathcal{V})(\one)
\]
the corresponding $\otimes$-invertible object, the ``Thom space'' of $\mc{V}$. With this notation, we have a group homomorphism $\mthom:K_{0}(X)\to \Pic(C(X))$ where $K_{0}(X)$ is the Grothendieck group of vector bundles on $X$.

If $V=\aone_X$ is the free rank~1 bundle then this equivalence is the subject of the axiom \ref{ax:stable} and we write $(-)(1):=\thom(\aone_{X})[-2]$.
More generally, one has `Tate twists' $(n)$ for arbitrary $n\in\Z$.

For $f:Y\to X$ \emph{smooth}, with locally free sheaf of relative differentials $\Omega_{f}$, there are equivalences of functors
\begin{equation}
\label{eq:relative-purity}
f^!\simeq \thom(\Omega_f)f^*
\end{equation}
and
\[
f_{!}\simeq f_{\sharp}\thom^{-1}(\Omega_{f}).
\]

Note that every morphism of $\base$-schemes is, locally on the domain, the composition of a closed immersion and a smooth morphism.
As $C^!$ is a Zariski-sheaf, the properties~\eqref{eq:relative-purity} and \Cref{rmk:localization}.\ref{it:localization.i-upper-shriek} therefore determine the functor~$C^!$ essentially uniquely.
\subsubsection{Exterior products and K\"unneth formula}\label{sec:ext-prod}

Given two schemes $X_i$, $i=1,2$, denote by $p_i:X_1\times_{\base} X_2\to X_i$ the canonical projection.
We denote the external product by
\[
M_1\boxtimes M_2 = p_1^*M_1\otimes p_2^*M_2
\]
for $M_i\in C(X_i)$.
If $f_i:X_i\to Y_i$ are $\base$-morphisms, $i=1,2$, then we have an equivalence in $C(Y_1\times_{\base} Y_2)$:
\[
(f_1\times f_2)_!(M_1\boxtimes M_2)\isoto (f_1)_!M_1\boxtimes (f_2)_!M_2.
\]
In fact, the exterior products also provide an alternative way to encode the symmetric monoidal structure on the categories $C(X)$ for varying $X$. To see this, we use the following lemma:

\begin{lem}{\cite[Theorem 2.4.3.18]{HA}}\label{lem:ext-prod}
  Let $C$ be an \icat with finite coproducts and $D^{\otimes}$ be a symmetric monoidal \icat. There is a canonical equivalence of \icats
  \[
\Fun^{\lax}(C^{\coprod},D^{\otimes})\simeq \Fun(C,\calg{D^{\otimes}})
\]
where $\Fun^{\lax}(-,-)$ denotes the category of symmetric lax-monoidal functors between two symmetric monoidal categories.
\end{lem}  

\cref{lem:ext-prod} and the interplay between Cartesian and coCartesian monoidal structures on opposite \icats imply that a coefficient system $C$ can be equivalently described as a symmetric lax-monoidal functor
\[C^{\boxtimes}:\schop[\base]\to \Cat{st} \]
satisfying certain properties. Informally, this lax-monoidal structure can be described as follows.
For any two $\base$-schemes~$X_1$ and~$X_2$, the external product is exact in both arguments separately and therefore induces an exact functor on the tensor product in~$\Cat{st}$,
\begin{equation}
\label{eq:boxtimes}
\boxtimes:C(X_1)\otimes C(X_2)\longrightarrow C(X_1\times_{\base} X_2),
\end{equation}
and the unit is encoded by an exact functor on finite spectra (the unit of~$\Cat{st}$) $\Sp^{\omega}\to C(\base)$ which sends the sphere spectrum to the object $\one_{C(\base)}$.
Conversely, given coherent exterior products as in~\eqref{eq:boxtimes}, the `internal' tensor product on~$C(X)$ can be recovered as the composite
\[
C(X)\times C(X)\xto{\boxtimes}C(X\times_{\base}X)\xto{\Delta^*}C(X)
\]
where $\Delta:X\to X\times_{\base}X$ is the diagonal embedding.

\subsection{Miscellanea}
\label{sec:more-axioms}

\subsubsection{Triangulated coefficient systems}

The notion of coefficient system is an \icategorical version of the closed symmetric monoidal stable homotopy $2$-functors of~\cite[Definitions 1.4.1, 2.3.1, and 2.3.50]{ayoub:thesis-1}.
The enhancement here is rather powerful as it allows one to obtain a more complete six-functor formalism (\cref{sec:cs-properties}) and to extend in some cases the coefficient system to equivariant or stacky settings. For example, the exceptional functoriality exists for all morphisms of $\base$-schemes, whereas~\cite{ayoub:thesis-1} could establish it only for quasi-projective morphisms.
In order to have this functoriality at least for \emph{separated} morphisms as well, \cite{cisinski-deglise:trig-cats} introduced the notion of motivic triangulated categories which includes one additional axiom. We will use this latter notion, with a slightly different terminology to emphasize the parallel with coefficient systems.

Let us write $\mTri$ for the $(2,1)$-category whose objects are symmetric monoidal triangulated categories, morphisms are monoidal exact functors and $2$-morphisms are invertible monoidal exact natural transformations.

\begin{defn}
\label{defn:triang-cs}
A \emph{triangulated coefficient system (over $\base$)} is a $2$-functor
$D:\schop[\base]\to \mTri$ satisfying the following two conditions:
\begin{enumerate}[(1)]
\item The triangulated analogues of the properties in \Cref{defn:cosy}. To be precise, these are \ref{Ax:left}, \ref{Ax:right}, \ref{Ax:hty-stable} of \Cref{defn:cosy}, together with \ref{it:localization-axiom.2} of \Cref{sta:localization-axiom}.
\item \namedlabel{adj}{(Adjoint)} If $f:Y\to X\in\sch[\base]$ is proper then $f_*$ admits a right adjoint.
\end{enumerate}
A \emph{morphism of triangulated coefficient systems} is a natural transformation between triangulated coefficient systems which satisfies the triangulated analogue of \Cref{defn:cs-morphism}.
\end{defn}  

Triangulated coefficient systems have the advantage that they can be manipulated without the heavy machinery of $\infty$-category theory. Moreover, as a rule of thumb, statements about (symmetric monoidal) stable $\infty$-categories which involve checking properties rather than constructing new structure often reduce immediately to statements about their (symmetric monoidal) triangulated homotopy categories. Here are some instances of this rule of thumb which we will use.

\begin{lem}
\label{sta:cosy-vs-triangulated}%
  Let $C:\schop[\base]\to\calg{\Cat{st}}$ be a functor.
\begin{enumerate}
\item \label{it:cosy-vs-triangulated.1} By passing to the homotopy categories, one gets a functor $\h{C}:\schop[\base]\to\mTri$.
\item \label{it:cosy-vs-triangulated.2} The functor $C$ is a coefficient system if and only if $\h{C}$ is a triangulated coefficient system.
\item \label{it:cosy-vs-triangulated.3} Let $C\to C'$ be a natural transformation of functors $\schop[\base]\to\calg{\Cat{st}}$. Then it is a morphism of coefficient systems if and only if the induced natural transformation $\h{C}\to \h{C}'$ is a morphism of triangulated coefficient systems.
\end{enumerate}  
\end{lem}  
\begin{proof}
See~\cite[Remark 5.10, Proposition 5.11, Corollary 5.13]{drew:motivic-hodge-modules} and~\cite[Remark~7.6]{drew-gallauer:usf} for~\ref{it:cosy-vs-triangulated.1}, \ref{it:cosy-vs-triangulated.3}, and \ref{it:cosy-vs-triangulated.2} without the additional axiom~\ref{adj}.
If $C$ is a coefficient system then $\h{C}$ automatically satisfies the latter since $f_*\simeq f_!$ has the right adjoint~$f^!$.
\end{proof}

\begin{rmk}
\label{rmk:triangulated-cosy-six-ops}
Given a triangulated coefficient system~$D$ over~$\base$, the properties listed in \Cref{sec:cs-properties} hold all true in~$D$---with the sole provision that the exceptional functoriality~$f_!$ and~$f^!$ exists only if $f$ is \emph{separated}.
This is shown in~\cite[Theorem~2.4.50]{cisinski-deglise:trig-cats}.
\end{rmk}

\subsubsection{Shifting}
\label{exa:shifted-cosy}
Fix a scheme~$X\in\sch[\base]$.
We denote by $C_X:\schop[\base]\to\Cat{st}$ the shifted coefficient system, that is informally defined by
\[
C_X(Y)=C(Y\times_\base X).
\]
Given a morphism $f:Y\to X$, evaluation at $f$ defines a morphism of coefficient systems $f^*:C_X\to C_Y$.
Assume now that $f$ is a proper morphism \resp{smooth morphism}.
Then the right \resp{left} adjoints define a natural transformation of functors $\schop[\base]\to\Cat{st}$
\[
f_*:C_Y\to C_X\qquad (\text{resp.\ }f_\sharp:C_Y\to C_X).
\]
This follows from proper \resp{smooth} base change (for which see \Cref{cs-basechange} below) and~\cite[Corollary~4.7.4.18.(3)]{HA}. 

In fact, when viewing $C_Y$ as a $C_X$-module, the latter being a commutative algebra object in $\fun{\schop[\base]}{\Cat{st}}$, \Cref{rmk:smooth-projection-formula} yields that $f_*$ \resp{$f_\sharp$} may be upgraded to a morphism of $C_X$-modules.

\subsubsection{Correspondences and the six-functor formalism}
\label{sec:corr-six}

An insight due to Lurie, and independently to Hörmann~\cite{hormann:six}, is that the six-functor formalism can be very compactly encoded in terms of categories of correspondences. Lurie's idea has been developed extensively by Gaitsgory and Rozenblyum in~\cite{GR1}, as well as by Liu-Zheng in~\cite{liu-zheng:6ff-artin-stacks}. We summarize the approach from~\cite{GR1}, completed by~\cite[Appendix~A]{richarz-scholbach:motivic-satake} for some monoidal aspects. This is to date the most powerful approach but unfortunately relies on unproven statements about $(\infty,2)$-categories (collected in~\cite[Chapter 10, 0.4.2]{gaitsgory-rozenblyum:vol1}).
For this reason, and as mentioned in the introduction, our `official' approach to exponentiation in this paper will not rely on correspondences.
Consequently, we allow ourselves to be brief and informal here (as well as in \cref{sec:conv-corr}), and present a picture which we cannot completely justify at the moment. The recent work of~\cite{mann:padic} may in the near future make the $(\infty,1)$-categorical aspects of the following completely rigourous.

The starting point is an $(\infty,2)$-category $\Corr(\sch[\base])$ of correspondences of $B$-schemes (see \cite[\S\,7.1.2]{GR1} for the construction). An object of $\Corr(\sch[\base])$ is a $\base$-scheme, a $1$-morphism from~$Y$ to~$X$ is a correspondence, i.e. two morphisms $Y\leftarrow Z \to X$, and a $2$-morphism $(Y\leftarrow Z \to X)\to (Y \leftarrow Z' \to X)$ is given by a commutative diagram:
\[
\begin{tikzcd}
    & Z \ar[ld] \ar[rd] \ar[dd,"p"'] & \\
    Y & & X \\
   & Z' \ar[lu] \ar[ru]  &
\end{tikzcd}
\]
The composition of correspondences is given by fiber product. The $(\infty,2)$-category $\Corr(\sch[\base])$ admits a symmetric monoidal structure~\cite[Chapter IX, \S 2.1.3]{GR1} which on objects is just given by fiber product over $B$, and there is a symmetric monoidal functor
\[
\iota^{*}:\schop[\base]\to \Corr(\sch[\base])
\]
(with the source equipped with the coCartesian symmetric monoidal structure) informally defined by $\iota^*(X)=X$ and $\iota^*(f:Y\to X)=(X\xleftarrow{f} Y \xto{\id}Y)$. We also write $\Corr(\sch[\base])^{\propsub}$ for the wide symmetric monoidal subcategory of $\Corr(\sch[\base])$ where we restrict to those $2$-morphisms with $p$ proper, and $\Corr(\sch[\base])^{\propsub,2\textup{-op}}$ for the same subcategory where we reverse the direction of the $2$-morphisms. 

As explained in \cref{sec:ext-prod}, the data of a presentable coefficient system~$C$ is equivalent to a symmetric lax-monoidal functor
\[C^{\boxtimes}:\schop[\base]\to \Cat{st} \]
satisfying certain properties. The outcome of the Lurie-Gaitsgory-Rozenblyum-$\ldots$ approach to six-functor formalisms is an extension of $C^{\boxtimes}$ along $\iota^*$ to a symmetric lax-monoidal $(\infty,2)$-functor\footnote{We are implictly using that $\Cat{st}$ has a natural $(\infty,2)$-category structure whose $2$-morphisms are possibly non-invertible natural transformations.}
  \[
C^{\boxtimes,*}_{!}:\Corr(\sch[\base])^{\propsub,2\textup{-op}}\to \Cat{st}
\]
sending a correspondence $\alpha:X\xleftarrow{f}Y\xto{\id} Y$ to~$f^*$ and $\beta:Y\xleftarrow{\id}Y\xto{f}X$ to~$f_!$ and such that $f_{!}$ has a right adjoint $f^{!}$ for all $f$. Since $C$ is a coefficient system, we already know that $f^{*}$ has a right adjoint $f_{*}$ and that $C(X)^{\otimes}$ has internal Homs, so that $C^{\boxtimes,*}_{!}$ encodes all six operations. The properties of these functors from \cref{sec:cs-properties} can then be recovered from the geometry of correspondences described in $\Corr(\sch[\base])^{\propsub,2\textup{-op}}$, as explained in \cite[Introduction of Part III]{GR1}.

For example, if $f$ is proper then the correspondence~$\alpha$ is left adjoint to~$\beta$ in the $(\infty,2)$-categorical sense and hence we obtain an equivalence $f_!\simeq f_*$. This, together with the requirement that $j_{!}$ is a left adjoint of $j^{*}$ for $j$ an open immersion and a few other properties which we do not make explicit here, should in fact essentially uniquely determine the $(\infty,2)$-functor $C^{\boxtimes,*}_{!}$ (fundamentally because these two cases determine $f_{!}$ for any $f$ by Zariski descent and Nagata's compactification theorem). Similarly, base change in the form of~$f'_!(g')^*\simeq g^*f_!$ associated to the Cartesian square~\eqref{eq:linearity-square} is a consequence of the fact that the correspondence
\[
Y\xleftarrow{g'}Y'\xto{f'}X'
\]
may be written as a composite in two different ways.

Let $\corr(\sch[\base])$ denote the $(\infty,1)$-category obtained from $\Corr(\sch[\base])$ by discarding all non-invertible $2$-morphisms. The extension result above implies in particular that there is an extension of $C^{\boxtimes}$ to a symmetric lax-monoidal $(\infty,1)$-functor which we will also denote by
$C^{\boxtimes,*}_{!}:\corr(\sch[\base])^{(\infty,1)}\to \Cat{st}$
and which satisfies similar properties. In many applications (such as the construction of the convolution product discussed in \cref{sec:conv-corr} below) this $(\infty,1)$-functor is enough; however, in the approach of \cite{GR1}, it seems that this $(\infty,1)$-categorical extension cannot be constructed without first passing through the $(\infty,2)$-categorical version.

\section{Exponentiation as a triangulated coefficient system}
\label{sec:exp-htpy}

The main construction of this paper associates to a coefficient system $C$ another coefficient system $\ep{C}$, the \emph{exponentiation} of $C$. In the present section we first perform this construction for triangulated coefficient systems, while in \Cref{sec:cs-icats} we will work at the level of \icats.
Many results about the latter construction will be formal consequences of the results in this section.
We proceed in this way in order to keep most of the discussion elementary and accessible to readers who don't want to wade into technicalities about symmetric monoidal \icats.

So, throughout this section we fix a triangulated coefficient system $D$ over $\base$.
The construction of $\ep{D}$ depends on a choice of smooth commutative $\base$-group scheme $\grp$ satisfying certain conditions. By far the most important case, which motivates the construction and in which those conditions are always satisfied, is the additive group $\grp=\Ga{\base}$. We encourage the reader to always keep this case in mind.

\begin{covn}\label{conv:group-scheme}
  We fix a commutative $\base$-group scheme $(\grp,\plus,\unit)$ where $\plus:\grp\times_\base\grp\to\grp$ and $\unit:\base\to\grp$ denote the multiplication and unit. We write $\pi:\grp\to \base$ for the structure map. For $X\in \sch[\base]$, we write $\pi_{X}:X_{\grp}:=X\times_{\base} \grp\to X$, and often also denote $\pi_{X}$ as $\pi$ when there is no confusion.
We make the following assumptions on $\grp$:
\begin{enumerate}
\item $\pi:\grp\to\base$ is separated\footnote{This extremely mild assumption is an artifact of \Cref{rmk:triangulated-cosy-six-ops} and can be removed if~$D$ underlies a coefficient sytem.} and smooth, and
\item the functor $\pi_{X}^*:D(X)\to D(X_{\grp})$ is fully faithful for all $X\in\sch$.
\end{enumerate}
\end{covn}

\begin{rmk}
These assumptions are satisfied for $\grp=\Ga{\base}$ by the axiom \ref{ax:hty} of coefficient systems.
Similarly they are satisfied for powers $\Ga{\base}^{n}$. The main other examples we have in mind are commutative unipotent algebraic groups over perfect fields of positive characteristic.
If one relaxes the assumption that $\grp$ is commutative, one obtains a similar theory except that the convolution product is only monoidal instead of symmetric monoidal.
\end{rmk}
\subsection{Definition through vanishing (co)homology}
\label{sec:exp-subcats}

As explained in the introduction, we would like to define the exponentiation as the Verdier quotient
\[
D(X_\grp)/\pi^*D(X).
\]
For technical reasons, it is very useful to realize the resulting category as a full subcategory of $D(X_\grp)$. We thus choose the following as our definition.
\begin{defn}\label{defn:hexp}
Let $X\in\sch[\base]$. We define $\ep{D}(X)\subseteq D(X_\grp)$ as the full subcategory given as the kernel of $\pi_{X!}:D(X_\grp)\to D(X)$.
Since $\pi_!$ is an exact functor, this is automatically a triangulated subcategory.
\end{defn}
In this subsection, we clarify the relationship between this definition and the Verdier quotient above, and provide several other equivalent characterizations of $\ep{D}(X)$. 
This is standard localization theory for triangulated categories.
We also introduce a coreflector~$\Pi$ which will play an important role in the sequel.
\begin{lem}
\label{sta:exp-recollement-*}%
There is a recollement of triangulated categories
\[
\begin{tikzcd}[column sep=large]
D(X)
\ar[r, "\pi^*"{description}]
&
D(X_\grp)
\ar[l, "{\pi_\sharp}" above, shift right=4]
\ar[l, "{\pi_*}" below, shift left=4]
\ar[r, "{Q}"{description}]
&
D(X_\grp)/\pi^*D(X)
\ar[l, "{Q_\lambda}" above, shift right=4]
\ar[l, "{Q_\rho}" below, shift left=4]
\end{tikzcd}
\]
in which
\begin{enumerate}
\item any functor is left adjoint to the one just beneath it,
\item $Q_\lambda$ is fully faithful and induces an equivalence
\[
D(X_\grp)/\pi^*D(X)\isoto \ker(\pi_\sharp) = {}^\bot\left(\pi^*D(X)\right)
\]
onto the left orthogonal complement, and
\item $Q_\rho$ is fully faithful and induces an equivalence
\[
D(X_\grp)/\pi^*D(X)\isoto\ker(\pi_*) = \left(\pi^*D(X)\right)^\bot
\]
onto the right orthogonal complement.
\end{enumerate}
\end{lem}
\begin{proof}
The functor $\pi^*:D(X)\to D(X_\grp)$ has a left adjoint $\pi_\sharp$ (since $\pi$ is smooth) and a right adjoint $\pi_*$.
By assumption, it is also fully faithful so applying~\cite[Proposition~4.9.1]{krause:localization-theory-tricat} and its dual gives the result.
(See also~\cite[Proposition~4.13.1]{krause:localization-theory-tricat}.)
\end{proof}
\begin{rmk}
\label{rmk:exp-recollement-!}%
The ordinary and exceptional pullback functors $\pi^*,\pi^!:D(X)\to D(X_\grp)$ are closely related.
Indeed, we have~(\Cref{cs-relative-purity})
\[
\pi^!\simeq \thom(\Omega_{\pi})\circ\pi^*
\]
where $\thom(\Omega_\pi)$ denotes the Thom equivalence with respect to the relative cotangent bundle $\Omega_\pi=\Omega_{X_\grp/X}$ of~$\pi_X$. 
Since $\grp$ is a smooth group scheme, it follows that this relative cotangent bundle is in fact the pullback of the conormal sheaf to the embedding $\unit:\base\to\grp$.
In particular we have
\begin{equation}
\label{eq:pi!=pi*}
\pi^!\simeq \thom(\Omega_{\pi})\circ\pi^*\simeq\pi^*\circ\thom(\left(\unit^*\Omega_{\grp/\base}\right)|_X)
\end{equation}
and it follows that the images of $\pi^!$ and $\pi^*$ coincide.
(In the special case of $\grp=\Ga{}$, this Thom equivalence can be identified with a Tate twist and shift: $\thom(\Omega_{\Ga{}})\simeq (-)(1)[2]$.)
We also deduce:
\end{rmk}

\begin{cor}
\label{sta:ker!=kersharp}
We have as triangulated subcategories of $D(X_{\grp})$:
\[
\ep{D}(X)=\ker(\pi_!)=\ker(\pi_\sharp)={}^\bot\left(\pi^*D(X)\right)
\]
\end{cor}
\begin{proof}
The first identity is \Cref{defn:hexp}, the second identity is obtained from~\eqref{eq:pi!=pi*} by adjunction, and the last one is \Cref{sta:exp-recollement-*}.
\end{proof}

\begin{rmk}
\label{rmk:ker-!-justification}%
While the triangulated categories $\ker(\pi_*)$ and $\ker(\pi_!)=\ker(\pi_\sharp)$ are equivalent, they are \emph{not} equal as subcategories of $D(X_\grp)$.
In our construction of exponentiation we have chosen to give preference to the realization of $\ep{D}(X)$ as $\ker(\pi_!)$.
The reason is that coefficient systems encode the ordinary pullback functors (or, $*$-functoriality) and while for $f:Y\to X$, the induced functor
\[
f_\grp^*:D(X_\grp)\to D(Y_\grp)
\]
restricts to the respective kernels of $\pi_!$ (\Cref{sta:hexp-subfunctor}), the same is not true for the kernels of~$\pi_*$.
This makes some formulas and arguments easier.

We note that this contrasts with~\cite{ks-coha}, where Kontsevich and Soibelman define exponential mixed Hodge structures as the kernel of~$\pi_{*}$ instead (with $X=\base=\Spec(\CC)$).
\end{rmk}

\begin{lem}
\label{sta:Pi-coreflector}%
The endofunctor $Q_\lambda\circ Q$ on $D(X_{\grp})$ factors through $\ep{D}(X)$ and induces a right adjoint
\[
\Pi:D(X_\grp)\to \ep{D}(X)
\]
to the inclusion.
Moreover, there is a distinguished triangle of endofunctors on $D(X_{\grp})$:
\begin{equation}
\label{eq:triangle-Pi}
\Pi\to \id\to\pi^*\pi_\sharp
\end{equation}
\end{lem}

\begin{rmk}
Here, as often in the sequel, we will abuse notation and identify the \emph{coreflector}~$\Pi$ with the colocalization $Q_\lambda\circ Q$.
As we will see below (\Cref{sta:coprojector-as-convolution}), this functor may also be expressed explicitly as the convolution product with a certain object in $\ep{D}(X)$.
\end{rmk}

\begin{proof}[Proof of \Cref{sta:Pi-coreflector}]
This is a formal consequence of the composite~$\pi^*\pi_\sharp$ being a localization functor, see~\cite[\S\,4.11]{krause:localization-theory-tricat}.
\end{proof}

\begin{rmk}
\label{rmk:ker-!-verdier-duality}%
Suppose $D$ affords a Verdier duality, that is, for each $X\in\sch[\base]$, there is an anti-involution
\[
\mathbb{D}_X:D(X)\isoto D(X)\op
\]
which exchanges the $*$-functoriality and the $!$-functoriality:
\[
\mathbb{D}_Yf^*\simeq f^!\mathbb{D}_X
\]
for every separated $f:Y\to X$.
Then $\mathbb{D}_X$ sends $\ker(\pi_!)$ to $\ker(\pi_*)$ and restricts to an antiequivalence
\[
\ker(\pi_!)\isoto\ker(\pi_*)\op.
\]
\end{rmk}

\subsection{Four-functor formalism}\label{sec:hfour}

In this section, we construct the non-monoidal part of the triangulated coefficient system~$\ep{D}$.
By definition, $\ep{D}(X)\subseteq D_{\grp}(X)$, where we use the shift $D_{\grp}$ of $D$ by $\grp$, cf.\ \Cref{exa:shifted-cosy}.
Among other things we will see that $\ep{D}\subseteq D_\grp$ is in fact a subfunctor, the `left adjoint' functoriality on $D_\grp$ restricts to $\ep{D}$, and the `right adjoint' functoriality descends \textsl{via} the coreflector~$\Pi$.

We will use (often without mentioning) the basic properties of coefficient systems described in \Cref{sec:cs-properties}, which as discussed in \ref{rmk:triangulated-cosy-six-ops} hold for triangulated coefficient systems with minor restrictions. The first easy important observation is the following result.

\begin{lem}\label{Pi-operations}
The endofunctor $\Pi$ commutes with pullbacks, $\sharp$-pushforwards and $!$-pushforwards in~$D_{\grp}$ in the following sense. Let $f:Y\to X$ be a morphism in $\sch[\base]$.
\begin{enumerate}
  \item\label{it:Pi-operations.1} The natural transformation $(f_{\grp})^{*}\Pi\to \Pi (f_{\grp})^{*}$ is an isomorphism.
  \item\label{it:Pi-operations.2} Assume that $f$ is smooth. The natural transformation $(f_{\grp})_{\sharp}\Pi\to \Pi (f_{\grp})_{\sharp}$ is an isomorphism.
    \item\label{it:Pi-operations.3} Assume that $f$ is separated. The natural transformation $(f_{\grp})_!\Pi\to \Pi (f_{\grp})_!$ is an isomorphism.
  \end{enumerate}
\end{lem}
\begin{proof}
The functors $\pi^{*}$, $\pi_{\sharp}$ between $D$ and $D_{\grp}$ as well as the unit $\id\to \pi^{*}\pi_{\sharp}$ commute with arbitrary pullbacks, $\sharp$-pushforwards and $!$-pushforwards (using base change and the fact that $\pi$ is smooth in the latter case).
Hence all claims follow from the defining distinguished triangle~\eqref{eq:triangle-Pi}.
\end{proof}

\begin{cor}
\label{sta:hexp-subfunctor}
Let $f:Y\to X$ be a morphism in $\sch[\base]$.
We have $f_{\grp}^{*}\ep{D}(X)\subset \ep{D}(Y)$.
Hence $\ep{D}$ defines a subfunctor of $D_{\grp}:\schop[\base]\to\Tri$.
\end{cor}
\begin{proof}
This follows from \Cref{Pi-operations}.\ref{it:Pi-operations.1}.
\end{proof}

Our next goal is to verify the subset of the axioms of a triangulated coefficient system for~$\ep{D}$ which do not involve the monoidal structure.

\begin{prop}\label{exponentiation-left-axioms}
Let $f:Y\to X$ be a morphism in $\sch[\base]$.
\begin{enumerate}
\item \label{hep:sharp} If $f$ is smooth then the functor $(f_\grp)_\sharp$ restricts to a functor $(f_\grp)_\sharp:\ep{D}(Y)\to \ep{D}(X)$ left adjoint to $f_\grp^{*}:\ep{D}(X)\to \ep{D}(Y)$.
\item \label{hep:proper-push} If $f$ is proper then the functor $(f_\grp)_*$ restricts to a functor $(f_\grp)_*:\ep{D}(Y)\to \ep{D}(X)$ right adjoint to $f_\grp^{*}:\ep{D}(X)\to \ep{D}(Y)$.
\item \label{hep:left-hty} The functor $\ep{D}$ satisfies the (triangulated version of) the axioms
     \ref{ax:smooth-bc} and \ref{ax:hty}.
 \end{enumerate}
 \end{prop}
 \begin{proof}
 The functor $(f_\grp)_\sharp$ restricts as claimed by \Cref{Pi-operations}.\ref{it:Pi-operations.2}.
 It is a left adjoint to $f_\grp^{*}$ and satisfies the Beck-Chevalley condition for cartesian squares, since those properties are satisfied at the level of $D_{\grp}$. This proves part~\ref{hep:sharp} as well as~\ref{ax:smooth-bc}.
 Part~\ref{hep:proper-push} follows similarly from \Cref{Pi-operations}.\ref{it:Pi-operations.3} since $(f_\grp)_{!}\simeq (f_\grp)_*$.

Finally, for~\ref{ax:hty} we notice that it is equivalent to the counit $(p_\grp)_\sharp (p_\grp)^*\to\id$ being invertible hence this follows from \ref{ax:hty} for $D_{\grp}$.
 \end{proof}

 We turn to the construction of the other functors. Unlike the $*$\nobreakdash-pullbacks and the $\sharp$\nobreakdash-push\-forwards, the functor $(f_\grp)_{*}$ does not restrict directly from $D_{\grp}$ in general and we need to take recourse to the coreflector $\Pi:D_\grp\to \ep{D}$ of \Cref{sta:Pi-coreflector}.

\begin{prop}\label{pro:hep-cs-non-mon}
 The functor $\ep{D}$ satisfies the axioms \ref{ax:pushforwards}, \ref{ax:stable} and condition \ref{it:localization-axiom.2} of \Cref{sta:localization-axiom}.
\end{prop}
 \begin{proof}
For any $f$ in $\sch[\base]$, the composite $\Pi (f_\grp)_*$ is a right adjoint to $(f_\grp)^*$. This proves \ref{ax:pushforwards}.

We certainly have $\ep{D}(\emptyset)=0$. Let $i$ be a closed immersion with complementary open immersion $j$. The pair $(i_\grp^{*},j_\grp^{*})$ is conservative at the level of $D_{\grp}$, so it is also conservative for the subfunctor $\ep{D}$.
Moreover, the right adjoint to $i_\grp^*$ at the level of $\ep{D}$ is given by $(i_\grp)_*$, by \Cref{exponentiation-left-axioms}.\ref{hep:proper-push}.
Therefore, \ref{it:localization-axiom.2} of \Cref{sta:localization-axiom} for $\ep{D}$ follows from the corresponding axiom for $D_\grp$.

We now prove \ref{ax:stable}. Let $X\in\sch$.
By \Cref{exponentiation-left-axioms}.\ref{hep:sharp} and the previous paragraph, we need to show that the functor $(q_\grp)_{\sharp}(i_\grp)_{*}:\ep{D}(X)\to \ep{D}(X)$ is an equivalence.
But this follows from \ref{ax:stable} for $D_{\grp}$.
\end{proof}

\begin{rmk}
\label{rmk:Pi-morphism-cosy}%
\label{rmk:hCexp-cosy-nonmonoidal}%
At this point, combining \Cref{exponentiation-left-axioms} and \Cref{pro:hep-cs-non-mon}, we have shown that the functor $\ep{D}$ is a `non-monoidal triangulated coefficient system' (that is, a stable homotopy $2$-functor in the sense of~\cite[Definitions 1.4.1]{ayoub:thesis-1}).
\Cref{sta:hexp-subfunctor} and \Cref{Pi-operations}.\ref{it:Pi-operations.2} together imply that $\Pi:D_{\grp}\to \ep{D}$ is in fact a morphism of `non-monoidal triangulated coefficient systems', that is, it satisfies the Beck-Chevalley condition in \Cref{defn:cs-morphism}.

By~\cite[Scholie 1.4.2]{ayoub:thesis-1}, we have a four-functor formalism for $\ep{D}$. To distinguish notationally from the operations in $D$, we adopt the following.
\end{rmk}

\begin{notn}
Let $f$ be a morphism in $\sch[\base]$. We write:
\begin{enumerate}[(a)]
\item $\ul{f}^{*}=\ep{D}(f)^*$
\item $\ul{f}_*=\ep{D}(f)_*$
\end{enumerate}
Assume $f$ is separated. We write:
\begin{enumerate}[(c)]
\item $\ul{f}_!=\ep{D}(f)_{!}$
\item $\ul{f}^!=\ep{D}(f)^{!}$
\end{enumerate}
Let $p:V\to X$ be a vector bundle with zero section $s:X\to V$. We write $\ul{\thom}(V)=\ul{p}_{\sharp}\ul{s}_{*}$ for the corresponding Thom equivalence.
\end{notn}

Let us identify the resulting operations.

\begin{prop}
\label{sta:identify-four-functors-hexp}
Let $f$ be a morphism in $\sch[\base]$.
  Then:
\begin{enumerate}[(a)]
\item $\ul{f}^{*}=f_\grp^*$
\item $\ul{f}_*=\Pi(f_\grp)_*$
\end{enumerate}
If $f$ is separated, then:
\begin{enumerate}[(a),resume]
\item $\ul{f}_!=(f_\grp)_!$
\item $\ul{f}^!=\Pi(f_\grp)^!$
\end{enumerate}
Moreover,
\begin{enumerate}[(a),resume]
\item if $f$ is smooth then $\ul{f}_\sharp=(f_\grp)_\sharp$,
\item if $f$ is smooth and separated then $\ul{f}^!=f_\grp^!$,
\item if $f$ is proper then $\ul{f}_*=(f_\grp)_*$ and
\item if $V$ is a vector bundle then $\ul{\thom}(V)=\thom(V_\grp)$.
\end{enumerate}
Finally, we always have $\Pi  (f_\grp)_{*} \Pi \simeq \Pi (f_\grp)_{*}$, and $\Pi  (f_\grp)^! \Pi \simeq \Pi  f_\grp^!$ for $f$ separated.
\end{prop}

\begin{proof}
The identification $\ul{f}^*=(f_\grp)^*$ is true by Lemma~\ref{sta:hexp-subfunctor}.
The claim about $f_\sharp$ if $f$ is  smooth \resp{$f_*$ if $f$ is proper} follows from Lemma~\ref{exponentiation-left-axioms}.
This implies the claim about Thom equivalences.
For $f$ smooth and separated, relative purity then shows $\ul{f}^!=f_\grp^!$.

For general separated $f$ choose a compactification $f=\overline{f}j$ with $\overline{f}$ proper and $j$ an open immersion, which exists by Nagata's theorem. Then $\ul{f}_{!}=\ul{\overline{f}}_{*}\ul{j}_{\sharp}$ which by the previous paragraph equals $(\overline{f}_{\grp})_{*}(j_{\grp})_{\sharp}=(f_{\grp})_{!}$. Passing to right adjoints this also gives $\ul{f}^!=\Pi f_\grp^!$.

The final statement follows by uniqueness of right adjoints.
\end{proof}

\subsection{Convolution product}
\label{sec:exp-htpy-convolution-shifted}

We now come to the monoidal structure. 

\begin{defn}\label{defn:conv}
Let $X\in\sch[\base]$. The \emph{convolution product} on $D_\grp(X)$ is the bifunctor 
\begin{align}
  \label{eq:conv-shifted}
  \conv\colon D_\grp(X)\times D_\grp(X)&\longrightarrow D_\grp(X)\\
\notag  (M,N)&\longmapsto \plus_{!}(M\boxtimes_{X} N).
\end{align}
\end{defn}

\begin{rmk}
\label{rem:conv-shifted-idea}%
Here is one way of thinking about the convolution product.
By a triangulated version of \Cref{lem:ext-prod}, $D$ defines a lax symmetric monoidal functor $\schop[\base]\to\Tri$.
The K\"unneth formula (\Cref{sec:ext-prod}) says precisely that the functor~$D$ with the covariant $!$-functoriality also defines a lax symmetric monoidal functor $D_!:\Sch[\base]{sep}\to\Tri$ where $\Sch[\base]{sep}$ has the same objects as $\sch[\base]$ but whose morphisms are required to be separated.
In particular, this functor takes the commutative separated $\base$-group scheme~$\grp$ to a commutative monoid in~$\Tri$, that is, a tensor triangulated category $D(\grp)$.
The tensor product is given by the convolution product~\eqref{eq:conv-shifted}.
The object $\unit_!\one\in D(\grp)$ is a unit for the convolution product.
And if $\tau:\grp^2\isoto\grp^2$ denotes the swapping of the two factors then the symmetry constraint for the convolution product is given by
\[
\plus_!(M\boxtimes N)\cong \plus_!\tau_!(N\boxtimes M)\cong (\plus\circ\tau)_!(N\boxtimes M)\cong\plus_!(N\boxtimes M).
\]
We deduce the following result which can also be proven by hand.
\end{rmk}

\begin{prop}
\label{sta:conv-shifted}
Let $X\in\sch[\base]$.
The convolution product~\eqref{eq:conv-shifted} underlies a symmetric monoidal structure on $D_\grp(X)$ described in \Cref{rem:conv-shifted-idea}.\qed
\end{prop}

In the rest of this section, we establish some basic properties of the convolution product that will simplify later computations.
\begin{lem}
\label{sta:conv-0_*}
Let $X\in\sch[\base]$. Let $M\in D_{\grp}(X)$ and $N\in D(X)$.
The K\"unneth formula  induces an isomorphism
  \[
 \unit_{*}N\conv M\isofrom \pi^{*}N\otimes M
\]
which is natural in $M$ and $N$. 
\end{lem}
\begin{proof}
We have
  \[
\unit_{*}N\conv M=\plus_!(\unit_{*}N\boxtimes M)\isofrom \plus_{!}(\unit\times\id)_*(N\boxtimes M)\simeq N\boxtimes M=\pi^{*}N\otimes M
\]
where the wrong-way isomorphism is an instance of the K\"unneth formula, and the third isomorphism uses $\plus\circ(\unit\times\id)=\id$.
\end{proof}

\begin{rmk}
\label{rmk:conv-shifted-natural}
The construction of \Cref{rem:conv-shifted-idea} that associates with the group~$\grp$ the symmetric monoidal \icat $D(\grp)$ is natural in~$\grp$.
In particular, given a morphism $f:\grp\to\grp'$ of $\base$-group schemes the functor $f_!:D(\grp)\to D(\grp')$ is symmetric monoidal with respect to the convolution product structure.
The following is a special instance for $f=\pi:\grp\to\base$.
\end{rmk}

\begin{lem}\label{lem:pi-conv}
  Let $X\in\sch[\base]$ and $M,N\in D_\grp(X)$. The K\"unneth formula yields isomorphisms
  \[
\pi_!(M\conv N)\isofrom \pi_!M\otimes\pi_!N
\]
and
\[
\pi_\sharp(M\conv N)\isofrom \pi_\sharp M\otimes\pi_!N
\]
  in $D(X)$, natural in~$M$ and~$N$. 
\end{lem}
\begin{proof}
We have
\[
\pi_!(M\conv N)= \pi_!\plus_!(M\boxtimes_X N)\simeq (\pi\times \pi)_!(M\boxtimes_X N) \simeq \pi_!M\otimes \pi_!N
\]
by the observation that $\pi\plus=\pi\times\pi$ and the K\"unneth formula~(\Cref{sec:ext-prod}). The second isomorphism then follows by relative purity for the smooth morphism $\pi$.
\end{proof}

\begin{rmk}
The construction of \Cref{rem:conv-shifted-idea} is also natural in~$D$.
That is, given a morphism of triangulated coefficient systems $\phi:D\to D'$ then $\phi_{\grp}:D(\grp)\to D'(\grp)$ is a symmetric monoidal functor with respect to the convolution products on both sides.
The following is a special instance for $\phi=f^*:D_{X}\to D_{Y}$.
\end{rmk}

\begin{lem}
\label{sta:f^*-conv}
Let $f:Y\to X$ be a morphism of $\base$-schemes.
The functor $f_{\grp}^*:D_{\grp}(X)\to D_{\grp}(Y)$ is symmetric monoidal with respect to the convolution product.
\end{lem}
\begin{proof}
This amounts to showing that $(-)^*$ commutes with all the operations that enter into the definition of the convolution product~\eqref{eq:conv-shifted} and the unit $\unit_!\one$.
But this is clear.
\end{proof}

In the study of the convolution product, it is useful to introduce some additional morphisms related to the group~$\grp$.

\begin{defn}\label{defn:shear}
  The two \emph{shear isomorphisms} are the two morphisms $\delta_1,\delta_{2}:\grp\times_{\base}\grp\simeq\grp \times_{\base}\grp$ of $\base$-schemes uniquely characterised by
  \[
\plus \delta_{1}=p_{1},\ p_{2}\delta_{1}=p_{2}\qquad\text{  and  }\qquad \plus \delta_{2}=p_{2},\ p_{1}\delta_{2}=p_{1},
  \]
  i.e. given informally by the formulas $\delta_{1}(x,y)=(x-y,y)$ and $\delta_{2}(x,y)=(x,y-x)$. We also write $d=p_{1}\delta_{1}$, i.e $d(x,y)=x-y$.
  \end{defn}

  \begin{lem}\label{lem:constant-conv}
Let $M\in D(X)$ and $N\in D_\grp(X)$. We have a canonical and functorial isomorphism
  \[
\pi^*M\conv N\simeq \pi^*(M\otimes \pi_{!}N).
\]
In particular, convolution with a constant object always produces a constant object.
\end{lem}  
\begin{proof}
The isomorphism is obtained as follows:
  \begin{align*}
    \pi^* M\conv N &=  \plus_!(p_1^*\pi^*M\otimes p_2^* N)\\
                   &\simeq  \plus_!(\plus^*\pi^*M\otimes p_2^* N)\\
                   &\simeq \pi^*M\otimes \plus_!p_2^*N\\
                   &\simeq \pi^*M\otimes \plus_!(\delta_1)_!\delta_1^*p_2^*N\\
                   &\simeq \pi^*M\otimes (p_1)_!p_2^*N\\
                   &\simeq \pi^*M\otimes \pi^*\pi_!N\\
    &\simeq \pi^*(M\otimes \pi_!N)
  \end{align*}
where we have used linearity (\Cref{cs-linearity}), the identities of \Cref{defn:shear} and the fact that, for any isomorphism~$g$, we have $\id\simeq g_!g^*$.
\end{proof}  

\begin{lem}\label{lem:Pi-const}
  Let $M\in D(X)$ and $N\in D_{\grp}(X)$. We have a canonical isomorphism
  \[
\pi^{*}M\otimes \Pi(N)\simeq \Pi(\pi^{*}M\otimes N)
  \]
\end{lem}
\begin{proof}
This follows from the definition of $\Pi$ and the projection formula.
\end{proof}  

\begin{lem}\label{lem:Pi-conv}
Let $M,N\in D_\grp(X)$. The natural transformation $\Pi\to\id$ induces isomorphisms
  \[
\Pi M\conv N\simeq \Pi(M\conv N)\simeq \Pi M\conv \Pi N.
\]
\end{lem}
\begin{proof}
Let us write $\theta:\Pi\to\id$ for the counit of the coreflector.
By \Cref{lem:constant-conv}, the cofiber of $\theta_M\conv \id:\Pi M\conv N\to M\conv N$ is constant so that $\Pi(\theta_M\conv \id)$ is an isomorphism.
On the other hand, by \Cref{lem:pi-conv}, the domain $\Pi M\conv N\in \ep{D}(X)$ so that $\theta:\Pi(\Pi M\conv N)\isoto\Pi M\conv N$.
Combining the two isomorphisms gives the first isomorphism in the statement.
The second follows by two applications of the first.
\end{proof}

Finally, we show that the convolution symmetric monoidal structure on $D_{\grp}$ is closed and describe its interaction with~$\Pi$.

\begin{lem}
\label{sta:conv-hom-shifted}
The convolution symmetric monoidal structure on~$D_{\grp}(X)$ is closed.
The internal Hom bifunctor is given by
\[
\Homint_{\conv}(M,N)\simeq (p_1)_*\Homint_{D_{\grp^2}(X)}(p_2^*M,\plus^!N).
\]
\end{lem}
\begin{proof}
This is a straightforward game of adjunctions.
\end{proof}

\begin{lem}\label{lem:Pi-Homint}
Let $M,N\in D_\grp(X)$. The natural transformation $\Pi\to\id$ induces a commutative square of isomorphisms
\[
\begin{tikzcd}
\Pi \Homint_{\conv}(M,\Pi N)
\ar[r, "\sim"]
\ar[d, "\sim"]
&
\Pi\Homint_{\conv}(M,N)
\ar[d, "\sim"]
\\
\Pi \Homint_{\conv}(\Pi M,\Pi N)
\ar[r, "\sim"]
&
\Pi \Homint_{\conv}(\Pi M,N)
\end{tikzcd}
\]
\end{lem}
\begin{proof}
This is a formal consequence of the tensor/internal Hom adjunction and of \Cref{lem:Pi-conv}.
\end{proof}  

\subsection{Exponentiated triangulated coefficient system}
\label{sec:exp-trcosy}
In this subsection we finally show that the convolution product lifts the functor $\ep{D}:\schop[\base]\to\Tri$ to a triangulated coefficient system.
The starting point is the observation (\Cref{lem:pi-conv}) that $\ep{D}(X)\subseteq D_{\grp}(X)$ is an ideal with respect to the convolution product.
It is then clear how to turn it into a symmetric monoidal category itself:
\begin{prop}
\label{sta:conv-on-hCexp}%
Let $X\in\sch[\base]$.
The convolution product of~\eqref{eq:conv-shifted} restricts to a bifunctor
\[
\conv:\ep{D}(X)\times \ep{D}(X)\to \ep{D}(X).
\]
It underlies a symmetric monoidal structure on $\ep{D}(X)$ with unit $\Pi\unit_!\one$.
Moreover, $\Pi:D_{\grp}(X)\to \ep{D}(X)$ is a symmetric monoidal functor.
\end{prop}
\begin{proof}
This is an instance of a general fact, for $D$ a symmetric monoidal category and $D'\subseteq D$ an ideal whose inclusion admits a retraction $G:D\to D'$.
We apply this to $\Pi:D_{\grp}(X)\to \ep{D}(X)$ using \Cref{lem:pi-conv}.
\end{proof}

\begin{lem}
\label{sta:C_G-tricosy}
The functor $D_{\grp}:\schop[\base]\to\Tri$ underlies a triangulated coefficient system $D_{\grp}^{\conv}:\schop[\base]\to\mTri$ where the symmetric monoidal structure is given by the convolution product.
\end{lem}
\begin{proof}
Note that the functor in the statement is a `non-monoidal' triangulated coefficient system since it is shifted from~$D$ (\Cref{exa:shifted-cosy}).
Let $f:Y\to X$ be a morphism of $\base$-schemes.
By \Cref{sta:f^*-conv}, the functor $f_{\grp}^*:D_{\grp}(X)\to D_{\grp}(Y)$ is symmetric monoidal.
It follows that we can lift the functor in the statement to a functor $\schop[\base]\to\mTri$ with respect to the convolution product.
It remains to verify \ref{ax:projection} and \ref{ax:closed}.
The latter was established in \Cref{sta:conv-hom-shifted} so we turn to \ref{ax:projection}.
Let $f:Y\to X$ be a smooth $\base$-scheme.
We want to show that the canonical morphism
\[
(f_{\grp})_\sharp(f_{\grp}^*M\conv N)\to  M\conv (f_{\grp})_\sharp N
\]
is invertible in $D(X)$, for any $M\in D_{\grp}(X)$ and $N\in D_{\grp}(Y)$.
Using the axioms~\ref{ax:left} for $D_{\grp}$ one easily reduces to showing that the canonical morphism
\[
(f_\grp)_\sharp(\plus_Y)_!\to (\plus_X)_!(f_{\grp\times\grp})_\sharp
\]
is invertible.
But $(\plus_Y)_!\simeq\thom(\left(\unit^*\Omega_{\grp/\base}\right)|_{Y_{\grp}}\circ(\plus_Y)_\sharp$ and similarly for~$X$.
The claim now follows immediately.
\end{proof}

\begin{thrm}
\label{sta:hCexp-cs}
Let $D$ be a triangulated coefficient sytem.
Then $\ep{D}$ with the convolution product of \Cref{sta:conv-on-hCexp} is a triangulated coefficient system.
Moreover, the natural transformation $\Pi:D_{\grp}^{\conv}\to \ep{D}$ is a morphism of triangulated coefficient systems.
\end{thrm}
\begin{proof}
Let $f:Y\to X$ be a morphism of $\base$-schemes.
By \Cref{sta:f^*-conv}, the functor $f_{\grp}^*:D_{\grp}(X)\to D_{\grp}(Y)$ is symmetric monoidal.
By \Cref{sta:hexp-subfunctor}, this functor restricts to $\ul{f}^*=f_{\grp}^*:\ep{D}(X)\to\ep{D}(Y)$.
It follows from the first statement in \Cref{sta:conv-on-hCexp} that this restricted functor is (at least) non-unital symmetric monoidal.
\Cref{sta:conv-on-hCexp} also describes the unit as $\Pi\unit_!\one$ so that unitality of $\ul{f}^*$ follows from \Cref{Pi-operations}.\ref{it:Pi-operations.1}.
As a consequence we have a functor $\ep{D}:\schop[\base]\to\mTri$ and it remains to check this is a triangulated coefficient system.

The `non-monoidal' axioms were established in \Cref{sec:hfour}, see \Cref{rmk:hCexp-cosy-nonmonoidal}.
It remains to verify \ref{ax:projection} and \ref{ax:closed}.
The latter follows from the corresponding axiom for $D_{\grp}$ established in \Cref{sta:conv-hom-shifted}, namely $\Homint_{\ep{D}(X)}(M,N)=\Pi\Homint_{\conv}(M,N)$.
The former also follows from the corresponding axiom for $D_{\grp}$, since $\ul{f}^*=f_{\grp}^*$ and $\ul{f}_\sharp=(f_{\grp})_\sharp$ by \Cref{sta:identify-four-functors-hexp}.

We now turn to the last statement.
By \Cref{sta:conv-on-hCexp}, $\Pi:D_{\grp}^{\conv}\to \ep{D}$ defines a natural transformation of functors $\schop[\base]\to\mTri$.
We already noticed in \Cref{rmk:hCexp-cosy-nonmonoidal} that it satisfies the Beck-Chevalley condition.
This finishes the proof.
\end{proof}

In the course of the proof we in effect gave a formula for the internal Hom in $\ep{D}$ which we now make explicit for later use.

\begin{notn}
  Let $X\in\sch[\base]$ and $N,P\in\ep{D}(X)$. We denote the internal hom of $N$ and $P$ in $\ep{D}(X)$ by
$\ep{\Homint}(N,P)$.
\end{notn}  

\begin{prop}\label{sta:int-homs}
Let $X\in\sch$ and $N,P\in\ep{D}(X)$. We have a canonical isomorphism
\[
\ep{\Homint}(N,P)\simeq \Pi (p_1)_*\Homint_{D_{\grp^{2}}(X)}(p_2^*N,\plus^!P)
\]
in $D(X_\grp)$.
\end{prop}
\begin{proof}
This follows from \Cref{sta:conv-hom-shifted} by applying~$\Pi$.
\end{proof}

Similarly we have an explicit formula for Thom spaces.

\begin{notn}
  Let $X\in\sch[\base]$ and $v\in K_{0}(X)$. We denote by $\thomexp(v):C(X)\stackrel{\sim}{\to}C(X)$ (resp. $\mthomexp(v)\in \Pic(C(X))$) the corresponding Thom equivalence (resp. Thom space) for $\ep{C}(-)$.
\end{notn}

\begin{lem}\label{lem:thom-exp}
Let $X\in\sch[\base]$ and $v\in K_{0}(X)$ a K-theory class. Then we have
\[
\mthomexp(v)\simeq\Pi \unit_{*}\thom(v).
\]
\end{lem}
\begin{proof}
  It suffices to prove this for a vector bundle $p:V\to X$ with zero-section $s:X\to V$. Then \cref{Pi-operations} implies that
 \begin{eqnarray*}
   \thomexp(V)& \simeq & \ul{p}_{\sharp}\ul{s}_{!}\E{0}\\
              & \simeq & (p_{\grp})_{\sharp}(s_{\grp})_{*}\Pi \unit_{*}\one \\
              & \simeq & \Pi (p_{\grp})_{\sharp}(s_{\grp})_{*} \unit_{*}\one \\
& \simeq & \Pi  \unit_{*}\thom(v).
\end{eqnarray*}
\end{proof}

\begin{prop}
\label{sta:D-Dexp}%
The composite $\Pi\circ \unit_*:D\to \ep{D}$ underlies a morphism of triangulated coefficient systems.
Moreover, for $X\in\sch[\base]$, the functor $\Pi\circ \unit_*:D(X)\to\ep{D}(X)$ is split-faithful as soon as $\Hom_{\base}(X,\grp\backslash\{\unit\})\neq\emptyset$.
(For example, this is always satisfied if $\grp=\Ga{}$.)
\end{prop}
\begin{proof}
By proper base change, $\unit_*$ defines a morphism of `non-monoidal triangulated coefficient systems' $D\to D_{\grp}$.
By \Cref{rmk:conv-shifted-natural}, $\unit_*:D(X)\to D_{\grp}(X)$ is also symmetric monoidal for the convolution product.
In other words, $\unit_*:D\to D_{\grp}^{\conv}$ is a morphism of triangulated coefficient systems.
The first statement now follows from the second statement in \Cref{sta:hCexp-cs}.

For the second statement we may replace $\base$ by~$X$ and therefore assume a point $g:\base\to\grp\backslash\{\unit\}$ over~$\base$.
The fundamental distinguished triangle~\eqref{eq:triangle-Pi} gives
\[
g^!\Pi\unit_*\to g^!\unit_*\to g^!\pi^!\pi_!\unit_*,
\]
in which the middle term vanishes and which therefore induces an isomorphism $\id\isoto g^!\Pi\unit_*[1]$.
In other words, the functor $g^![1]$ is a retract to~$\Pi\circ\unit_*$.
\end{proof}

\section{Exponentiation as a coefficient system}
\label{sec:cs-icats}
In this section we embark on the construction of the functor $\ep{(-)}$ defined on coefficient systems. Throughout we fix a coefficient system $C\in\CoSy{\base}$ and we adopt the same \cref{conv:group-scheme} on the group scheme $\grp$ as in the previous section.
To ease the already heavy notation in this section we leave the base scheme~$\base$ mostly implicit.
In particular, we will often write $\sch$ for $\sch[\base]$ etc.

The construction on underlying \icats is straightforward.

\begin{defn}
\label{defn:Cexp-cosy}
Fix $X\in\sch$ and denote by $\ep{C}(X)\subseteq C_\grp(X)$ the full \subicat spanned by the kernel of $\pi_!:C(X_\grp)\to C(X)$.
\end{defn}

Since $\pi_!$ is an exact functor, $\ep{C}(X)\subseteq C_{\grp}(X)$ is a stable \subicat.
The whole issue is to equip the \icat $\ep{C}(X)$ with a symmetric monoidal structure lifting the convolution product on $\hep{C}(X)$ given by
\[
M\conv N:=\plus_{!}(M\boxtimes_{X} N),
\]
and to lift the pullback functors $f^{*}$ to symmetric monoidal functors.

\subsection{Convolution product via correspondences}\label{sec:conv-corr}

We first explain a construction based on the approach to six-functor formalisms in \cref{sec:corr-six}.
As in \cref{sec:corr-six}, this is an informal discussion, and the formal construction which we will actually use in the definition of $\ep{C}$ starts in \cref{sec:exp-cs-convolution}.

Let $X\in\sch[\base]$. The commutative $X$-group scheme $\grp_{X}$, considered as a commutative monoid scheme, induces $\grp_{X}\in \calg{(\sch[X])^{\times}}$. Recall the $(\infty,1)$-category $\corr(\sch[X])$ of correspondences of $X$-schemes discussed at the end of \cref{sec:corr-six}. There is a symmetric monoidal functor
\[
\iota_{!}:(\sch[X])^{\times}\to\corr(\sch[X])
\]
which informally is the identity on objects and sends $f:Z\to Y$ to $(Z\stackrel{\id_{Z}}{\leftarrow}Z\stackrel{f}{\to}Y)$. Hence we get $\iota_{!}\grp\in \calg{\corr(\sch[\base])}$. The commutative algebra structure on $\iota_{!}\grp$ is given by the correspondence
\[
\iota_{!}\grp\times \iota_{!}\grp \stackrel{\sim}{\longleftarrow }\iota_{!}\grp^{2} \stackrel{\plus}{\longrightarrow} \iota_{!}\grp
\]
As discussed in \cref{sec:corr-six}, there should be a symmetric lax-monoidal $(\infty,1)$-functor extending $C^{\boxtimes}$:
  \[
C^{\boxtimes,*}_{!}:\corr(\sch[X])\to \Cat{st}
\]
Since symmetric lax-monoidal functors send commutative algebra objects to commutative algebra objects, this yields $C(\grp)\in \calg{\Cat{st}}$, i.e.\ a symmetric monoidal structure on the stable \icat $C(\grp_{X})=C_{\grp}(X)$. The functor $C^{\boxtimes,*}_{!}$ then sends the correspondence above to
\[
C(\grp_{X})\times C(\grp_{X})\stackrel{\boxtimes_{X}}{\longrightarrow} C(\grp^{2}_{X})\stackrel{\plus_{!}}{\longrightarrow} C(\grp_{X}),
\]
which is precisely the convolution product.
This produces thus an \icategorical lift of \cref{defn:conv}.

This construction is for a fixed $X$.
To obtain compatible symmetric monoidal structures on $C_{\grp}(-)$, with symmetric monoidal lifts of the pullbacks $f^{*}_{\grp}$, it would then be necessary to work with correspondences indexed by $X\in\sch[\base]$. The type of arguments necessary for this will actually also occur in the context of the second construction in \cref{sec:exp-icats-shifted} below.

It still remains to check that this convolution product on $C_{\grp}(-)$ restricts to $\ep{C}(-)$ and induces a symmetric monoidal structure there as well; the only complication there is that the unit objects are not the same. We will see the necessary arguments for this in the context of the second construction in \cref{sec:exp-icats-exp} below.

\subsection{Exterior convolution product}
\label{sec:exp-cs-convolution}

We now start our second, rigourous, approach to the convolution product. The first step is to construct the \icategorical counterpart of the exterior product $M\boxtimes_{X}N$, functorially in $X$.

\begin{notn}
We denote by
\[
\begin{tikzcd}
\int_{\schop}\schop
\ar[d, "t"]\\
\schop
\end{tikzcd}
\]
the coCartesian fibration associated to the functor $\schop\to\iCat$ sending a scheme $X$ to $\schop[X]$ and a morphism $f:X'\to X$ to the pullback $-\times_{X}X'$ along~$f$.
Alternatively, $\int_{\schop}\schop=\fun{\Delta^1}{\sch}\op$ is the opposite of the arrow category of~$\sch$ and the structure map of the coCartesian fibration is evaluation at~$1$.
\end{notn}

\begin{cns}
\label{cns:exp-unstraightening}%
Consider the composite
\[
\int_{\schop}\schop\xto{s}\schop\xto{\tilde{C}^{\otimes}}\calg{\iCat}
\]
of the `source' functor and the opposite of the given coefficient system: $\tilde{C}(X)=C(X)\op$ with the opposite symmetric monoidal structure~\cite[Example~2.7]{drew-gallauer:usf}.
By symmetric monoidal (un)straightening~\cite[Appendix~A]{drew-gallauer:usf}, this composite classifies a coCartesian fibration:
\[
\begin{tikzcd}
\tilde{\cC}^{\boxtimes}
\ar[d, "p"]
\\
\left(
\int_{\schop}\schop
\right)^{\amalg}
\end{tikzcd}
\]
\end{cns}

\begin{rmk}
Informally, the \icat $\tilde{\cC}$ may be described as follows.
Objects are pairs $(a:Y\to X, M)$ of morphisms~$a$ in $\sch$ and an object $M\in C(Y)$ on the source.
A morphism from $(a:Y\to X, M)$ to $(a':Y'\to X', M')$ is a morphism of arrows
\[
\begin{tikzcd}
Y
\ar[r, "a"]
&
X
\\
Y'
\ar[u, "f"]
\ar[r, "a'"]
&
X'
\ar[u]
\end{tikzcd}
\]
in $\sch$ together with a morphism $M'\to f^*M$ in~$C(Y')$.
The tensor product is given by the external product:
\[
(a,M)\boxtimes(a',M')=(a\times a':Y\times Y'\to X\times X',M\boxtimes M')
\]
\end{rmk}

\begin{notn}
Recall the commutative nonunital \ioperad $\Commnu\subseteq\Comm$~\cite[\S\,5.4.4]{HA}.
For any \ioperad $\cO^\otimes\to\Comm$ we will systematically use the notation $\cO^\otimes_\tnu$ to denote the fiber product $\cO^\otimes\times_{\Comm}\Commnu$.
\end{notn}

\begin{notn}
Fix a scheme $X$ and consider the non-full subcategory of $\sch[X\times \grp]$ spanned by \emph{smooth} $a:Y\to X\times \grp=X_\grp$ and whose morphisms are \emph{smooth}.
We denote it by $\sm{X}$.

Composition with $\pi:X\times \grp\to X$ defines a functor
\[
\sm{X}\to\sch[X]
\]
which we integrate to a morphism of coCartesian fibrations $\pi_*:\int_{\schop}(\sm{})\op\to\int_{\schop}\schop$ over $\schop$.
The domain admits a nonunital symmetric monoidal structure with \emph{convolution product}:
\[
Y\conv Y'=\left( Y\times Y'\xto{a\times a'}(X\times \grp)\times (X\times \grp)\xto{\simeq}(X\times X)\times(\grp\times \grp)\xto{\plus}(X\times X)\times \grp\right)
\]
and we see that the functor $\pi_*$ underlies a nonunital symmetric monoidal structure:
\begin{equation}
\label{eq:pi_*}
\begin{tikzcd}
\left(\int_{\schop}\smop{}\right)^{\circledast}
\ar[rr, "\pi_*"]
\ar[rd]
&&
\left(\int_{\schop}\schop\right)^{\amalg}_{\tnu}
\ar[ld, "t"]
\\
&(\schop)^{\amalg}_{\tnu}
\end{tikzcd}
\end{equation}
We let $\tilde{\cC}_{/\grp}^{\convbox}$ denote the fiber product
\[
\tilde{\cC}^{\boxtimes}\times_{\left(\int_{\schop}\schop\right)^{\amalg}}{\left(\int_{\schop}\smop{}\right)^{\circledast}
}.
\]
\end{notn}

\begin{rmk}
\label{rmk:tildeC-convbox-informal}%
Informally, the \icat $\tilde{\cC}_{/\grp}$ may be described as follows.
The objects are pairs $(a, M)$ where $a:Y\to X\times\grp$ is smooth and $M\in C(Y)$.
A morphism $(a,M)\to (a',M')$ is a commutative square
\[
\begin{tikzcd}
Y
\ar[r, "a"]
&
X\times\grp
\\
Y'
\ar[u, "f"]
\ar[r, "a'"]
&
X'\times\grp
\ar[u, "g\times\grp"]
\end{tikzcd}
\]
where $f$ is smooth, together with a morphism $M'\to f^*M$ in $C(Y')$.
The tensor product is given by
\[
(Y,M)\convbox(Y',M')=(Y\circledast Y',M\boxtimes M')
\]
where we suppressed the structure morphisms $a,a'$ of $Y,Y'$, respectively.
\end{rmk}

\begin{cns}
The coCartesian fibration  $\tilde{\cC}_{/\grp}^{\convbox}\to(\schop)^{\amalg}_{\tnu}$ is classified by a functor
\[
\schop\to\calgnu{\iCat}
\]
which we may compose with the symmetric monoidal equivalence which takes an \icat to its opposite.
The resulting functor classifies another coCartesian fibration, called the \emph{dual coCartesian fibration}, which we denote by $\cC_{/\grp}^{\convbox}\to(\schop)^{\amalg}_{\tnu}$.
\end{cns}

\begin{rmk}
\label{rmk:C-convbox-informal}%
According to~\cite{barwick-glasman-nardin:dual-fibrations}, the \icat $\cC_{/\grp}^{\convbox}$ may informally be described as follows.
Objects are pairs $(a,M)$ as in \Cref{rmk:tildeC-convbox-informal} and the tensor product also remains unchanged.
However, a morphism from $(a,M)$ to $(a',M')$ is a triple $(g,f,\alpha)$ where
$g:X'\to X$ is a morphism in $\sch$, $f:Y\times_X{X'}\to Y'$ is a smooth morphism over $X'_\grp$, and $\alpha:(g')^*M\to f^*M'$ is a morphism in $C(Y\times_XX')$, where we denote by $g':Y\times_XX'\to Y$ the base change of~$g$.
\end{rmk}

\subsection{Shifting and \texorpdfstring{$\sharp$}{\#}-convolution}
\label{sec:exp-icats-shifted}

The second step consists in constructing an \icategorical version of the variant of the convolution product given by the formula
\[
(M,N)\mapsto \plus_{\sharp}(M\boxtimes_{X}N).
\]

\begin{cns}
The coCartesian fibration $t\circ\pi_*:\int_{\schop}\smop{}\to\schop$ admits a section
\[
\Delta\times\grp:\schop\to\int_{\schop}\smop{}
\]
that sends a scheme~$X$ to the identity morphism $X_\grp\to X_\grp$.
The pullback of $\cC_{/\grp}$ along $\Delta_\grp$ is a coCartesian fibration $\cC_{\grp}\to\schop$ that is classified by the shifted `coefficient system' (ignoring the symmetric monoidal structure)
\[
C_{\grp}=C((-)_\grp):\schop\to\iCat,
\]
cf.\ \Cref{exa:shifted-cosy}.
\end{cns}

\begin{lem}
\label{sta:exp-leftadjoint}%
\begin{enumerate}[(a)]
\item
Let $X\in\sch$.
The inclusion $C(X_\grp)\into (\cC_{/\grp})_X$ admits a left adjoint.
We denote by $L_X:(\cC_{/\grp})_X\to(\cC_{/\grp})_X$ the corresponding localization functor.
\item
The collection of localization functors $(L_X)_{X\in\schop}$ is compatible with the $(\schop)^{\amalg}_{\tnu}$-monoidal structure in the sense of~\cite[Definition~2.2.1.6]{HA}.
\end{enumerate}
\end{lem}

\begin{rmk}
For the first part, let $M'\in C(X_\grp)$ and $(a:Y\to X_\grp,M)\in(\cC_{/\grp})_X$.
As follows from \Cref{rmk:C-convbox-informal}, a morphism $(a,M)\to(\Delta_\grp)(M')$ corresponds to a morphism $\alpha:M\to a^*M'$ in $C(Y)$.
As $a$ is smooth, this is equivalent to giving a morphism $a_{\sharp}M\to M'$ in $C(X_\grp)$.
To make this argument rigorous we will proceed as in~\cite[Lemma~3.18.(b)]{drew-gallauer:usf}.

Similarly, the second part essentially follows from \ref{ax:smooth-bc} and \ref{ax:projection}, as in~\cite[Lemma~3.18.(c)]{drew-gallauer:usf}.
\end{rmk}
\begin{proof}[Proof of \Cref{sta:exp-leftadjoint}]
Consider the composite
\[
\Delta^1\times\int_{\schop}\smop{}\xto{\textup{ev}}\schop\xto{\tilde{C}}\iCat
\]
and the corresponding morphism of coCartesian fibrations
\[
\begin{tikzcd}
\tilde{\cC}_{/\grp}
\ar[rd, "p"]
&&
\tilde{\cD}_{/\grp}
\ar[ll, "\Pi^*" above]
\ar[ld, "q"]
\\
&\int_{\schop}\smop{}
\end{tikzcd}
\]
which, over a fixed $a:Y\to X_\grp\in\sm{X}$ is classified by the symmetric monoidal functor $a^*:C(X_\grp)\op\to C(Y)\op$.
Since $a$ is smooth, the functor $a^*$ admits a right adjoint, and it follows~\cite[Corollary~7.3.2.7]{HA} that $\Pi^*$ admits a relative right adjoint.
We denote the relative right adjoint by $\Pi_\sharp$.

Let
\[
\begin{tikzcd}
\tilde{\cC}_{/\grp}
\ar[rd, "p_1"]
&&
\tilde{\cD}_{/\grp}
\ar[ll, "\Pi^*" above]
\ar[ld, "q_1"]
\\
&\schop
\end{tikzcd}
\]
be the diagram obtained by composing $p$ and $q$ with the coCartesian fibration $t\circ\pi_*$.
The fiber over $X\in\sch$ is easily seen to be the functor
\[
(\tilde{\cC}_{/\grp})_X\xfrom{\Pi^*} (\sm{X})\op\times C(X_\grp)\op
\]
that sends $(a:Y\to X_\grp,M)$ to $(a,a^*M)$.
Now, $\id_{X_\grp}$ is a final object of $\sm{X}$, that is, the inclusion $\id_{X_\grp}:\Delta^0\into(\sm{X})\op$ is a left adjoint.
It follows that the composite
\[
\Pi^*\circ\id_{X_\grp}:C(X_\grp)\op\to(\tilde{\cC}_{/\grp})_X
\]
is a left adjoint, with right adjoint~(using~\cite[Proposition~7.3.2.5]{HA})
\[
\Pi_{X,\sharp}:(\tilde{\cC}_{/\grp})_X\xto{\Pi_\sharp}(\sm{X})\op\times C(X_\grp)\op\to C(X_\grp)\op
\]
where the last functor is the canonical projection onto the second factor.
This proves the first part.

We now turn to the second part.
Let $f:X\to \prod_iX_i$ be a morphism in $\sch$ and let $g_i:(a_i,M_i)\to(a_i',M_i')$ be $\Pi_{X_i,\sharp}$-equivalences in $(\tilde{\cC}_{/\grp})_{X_i}$.
By definition, we need to show that $\otimes_f\{g_i\}$ is a $\Pi_{X,\sharp}$-equivalence.
It follows from \ref{ax:smooth-bc} that the functor $\Pi_{\sharp}$ preserves $\schop$-coCartesian edges thus we may assume $f=\id_{\prod_iX_i}$.
Moreover, by induction and symmetry we reduce to the case of two factors $X_1\times X_2$ and $g_2=\id_{(a_2,M_2)}$.
Identify $g_1$ with a smooth morphism $g:Y_1'\to Y_1$ over $(X_1)_\grp$ together with a morphism $\alpha:M_1'\to g^*M_1$ in $C(Y_1')$.
We need to show that the induced morphism
\[
\plus_{\sharp}(a_1'\times a_2)_{\sharp}(M_1'\boxtimes M_2)\to\plus_\sharp(a_1\times a_2)_{\sharp}(M_1\boxtimes M_2)
\]
is an equivalence.
This is true before applying $\plus_{\sharp}$ since the morphism identifies, using \ref{ax:smooth-bc} and \ref{ax:projection}, with
\[
\Pi_{X_1,\sharp}g_1\boxtimes\id:(a'_1)_{\sharp}M_1'\boxtimes (a_2)_{\sharp}M_2\to (a_1)_{\sharp}M_1\boxtimes (a_2)_{\sharp}M_2
\]
which is an equivalence by assumption.
\end{proof}
From~\cite[Proposition~2.2.1.9]{HA} we deduce the following statement.\footnote{Note there is a typo in the statement of said result: The target of the functor in part~(2) should be $\cO^\otimes$ instead of $\mathrm{N}(\fin)$.}
\begin{cor}
\label{sta:sharp-convolution}%
The shifted `coefficient system' underlies a functor $C_\grp^{\convsharp}:\schop\to\calgnu{\iCat}$ whose associated coCartesian fibration fits into a $(\schop)^{\amalg}_\tnu$-monoidal functor
\[
\begin{tikzcd}
\cC_{/\grp}^{\convbox}
\ar[rr]
\ar[rd]
&&
\cC_{\grp}^{\convboxsharp}
\ar[ld]
\\
&(\schop)_{\tnu}^{\amalg}
\end{tikzcd}
\]
which is left adjoint to the inclusion.\qed
\end{cor}

\begin{rmk}
Informally, the \icat $\cC_\grp$ may be described as follows.
An object is a pair $(X,M)$ where $X\in\sch$ and $M\in C(X_\grp)$.
A morphism $(X,M)\to (X',M')$ consists of a morphism $f:X'\to X$ in $\sch$ together with a morphism $\alpha:f^*M\to M'$ in $C(X')$.
The $\sharp$-convolution product is given by
\[
(X,M)\convboxsharp(X',M')=(X\times X',\plus_{\sharp}(M\boxtimes M')).
\]
\end{rmk}

\subsection{Twisting the \texorpdfstring{$\sharp$}{\#}-convolution}

The morphism $\plus$ is smooth, hence by relative purity (\cref{cs-relative-purity}) we have
\[
\plus_{!}=\plus_{\sharp}\thom^{-1}(\Omega_{\plus}).
\]
Moreover $\grp$ is a group scheme, which implies that $\Omega_{\plus}$ is pulled back from $\base$, so that we have
\[
\plus_{!}=\thom^{-1}(\Omega_\pi)\plus_{\sharp}.
\]
This suggests to define the convolution from the $\sharp$-convolution product of the previous section by twisting compatibly with inverse Thom equivalences. 

\begin{cns}
Consider the following diagram of solid arrows in the category of simplicial sets:
\[
\begin{tikzcd}
\Delta^0
\ar[r, "C_\grp^{\convsharp}"]
\ar[d, "0" left]
&
\fun{\schop}{\calgnu{\iCat}}
\ar[d]
\\
J
\ar[r, "\thom(\Omega_\pi)" below]
\ar[ru, dotted, "\phi"]
&
\fun{\schop}{\iCat}
\end{tikzcd}
\]
where:
\begin{itemize}
\item $J$ is the nerve of the category with two objects $0,1$ and a unique isomorphism between them,
\item the bottom horizontal arrow sends both objects to $C_\grp$, the edge $0\to 1$ to the Thom equivalence $\thom(\Omega_\pi)$ (see \Cref{sta:thom-functorial} below), as in \Cref{cs-relative-purity}, and the edge $1\to 0$ to the inverse Thom equivalence,
\item the top horizontal arrow is provided by \Cref{sta:sharp-convolution},
\item the right vertical arrow is a Joyal fibration representing the forgetful functor.
\end{itemize}
Since the left vertical arrow is a trivial cofibration for the Joyal model structure there exists a lift $\phi$ as indicated by the dotted arrow in the diagram.
Evaluating $\phi$ at $1$ yields a new functor $C_\grp^{\conv}:\schop\to\calgnu{\iCat}$.
\end{cns}

\begin{lem}
\label{sta:thom-functorial}%
Let $p:V\to \base$ be a smooth morphism with a section $s:\base\to V$, denote by $p_X$ and $s_X$ the pullbacks to $X\to \base$ for each~$X$.
The family of Thom equivalences $\thom(p_X,s_X):=(p_X)_\sharp(s_X)_*:C(X)\isoto C(X)$ underlies an autoequivalence of the functor $C:\schop\to\iCat$.
\end{lem}
\begin{proof}
Compare \Cref{exa:shifted-cosy}.
Let us view the morphism $p$ as a functor $\Delta^{1,\opname}\to\schop$ and consider then the composite
\[
\Delta^{1,\opname}\times\schop\xto{p\times\id}\schop\times\schop\xto{\times}\schop\xto{C}\iCat.
\]
By adjunction, it corresponds to a functor $\schop\to\fun{\Delta^{1,\opname}}{\iCat}$ which, we claim, factors through $\fun[\textup{LAd}]{\Delta^{1,\opname}}{\iCat}$ in the sense of~\cite[Definition~4.7.4.16]{HA}.
Indeed, that amounts to \ref{ax:smooth-bc} for the smooth morphisms $p_X:V_X\to X$.
By~\cite[Corollary~4.7.4.18.(3)]{HA}, passing to the left adjoints $(p_X)_\sharp$ results in another functor $\schop\to\fun[\textup{RAd}]{\Delta^1}{\iCat}$, thus a morphism $p_\sharp:C_V\to C$ in $\fun{\schop}{\iCat}$.

The argument for the existence of $s_*:C\to C_V$ is the same, using proper base change instead of \ref{ax:smooth-bc}.
Combining the two, we obtain a morphism $p_\sharp s_*:C\to C$ in $\fun{\schop}{\iCat}$ which is pointwise the Thom equivalence $\thom(p_X,s_X)$, thus an equivalence.
\end{proof}

By construction, the underlying functor $C_\grp$ of $C_\grp^{\conv}$ remains unchanged but the nonunital symmetric monoidal structure can informally be described as follows.
\begin{lem}
\label{sta:convolution-description}%
For $X\in\sch$ and $M,N\in C(X_\grp)$, we have
\[
M\conv N\simeq \plus_!(M\boxtimes N).
\]
\end{lem}
\begin{proof}
To see why that is true let us first describe the tensor product on $C_\grp^{\convsharp}(X)$.
Consider the following diagram with (hopefully) obvious notation
\[
\begin{tikzcd}
X_\grp
\ar[r, "\Delta"]
&
X^2_{\grp}
\\
X_{\grp^2}
\ar[r, "\Delta"]
\ar[u, "\plus"]
\ar[d, "p_1" left]
\ar[rd, "p_2" near start]
&
X^2_{\grp^2}
\ar[u, "\plus" right]
\ar[dl, "q_1" near end, crossing over]
\ar[d, "q_2"]
\\
X_\grp
&
X_\grp
\end{tikzcd}
\]
By construction, we then have in $C^{\convsharp}_\grp(X)$ (with $\otimes$ denoting the tensor product in $C$):
\begin{align*}
  M_1\convsharp M_2&
                     \simeq \Delta^*\plus_\sharp(q_1^*M_1\otimes q_2^*M_2)\\
                   &\simeq \plus_\sharp\Delta^*(q_1^*M_1\otimes q_2^*M_2)\\
                   &\simeq \plus_\sharp(p_1^*M_1\otimes p_2^*M_2)
\end{align*}
It follows that in $C^{\conv}_{\grp}(X)$ we have
\begin{align*}
  M_1\conv M_2 &
                 \simeq \thom(\Omega_\pi)\plus_\sharp\left(p_1^*\thom^{-1}(\Omega_\pi)M_1\otimes p_2^*\thom^{-1}(\Omega_\pi)M_2\right)\\
               & \simeq \thom(\Omega_\pi)\plus_\sharp\left(\thom^{-1}(p_1^*\Omega_\pi)p_1^*M_1\otimes \thom^{-1}(p_2^*\Omega_\pi)p_2^*M_2\right)\\
               & \simeq \thom(\Omega_\pi)\plus_\sharp\thom^{-2}(\plus^*\Omega_\pi)\left(p_1^*M_1\otimes p_2^*M_2\right)\\
               & \simeq \thom^{-1}(\Omega_\pi)\plus_\sharp\left(p_1^*M_1\otimes p_2^*M_2\right)\\
               & \simeq \plus_!\left(p_1^*M_1\otimes p_2^*M_2\right)
\end{align*}
as claimed. 
\end{proof}

\subsection{Exponentiated coefficient system}
\label{sec:exp-icats-exp}

Recall (\Cref{defn:Cexp-cosy}) that we denote by $\ep{C}(X)\subseteq C_\grp(X)$ the full \subicat spanned by the kernel of $\pi_!:C(X_\grp)\to C(X)$.

\begin{lem}
The association $X\mapsto \ep{C}(X)$ underlies a functor $\ep{C}:\schop\to\calgnu{\iCat}$ so that the inclusion $\ep{C}\subseteq C_\grp^{\conv}$ is nonunital symmetric monoidal.
\end{lem}
\begin{proof}
This follows from~\cite[Proposition~2.2.1.1]{HA} once we verify the following two conditions:
\begin{enumerate}
\item The \subicat $\ep{C}(X)\subseteq C_\grp(X)$ is stable under convolution product.
\item For $f:X'\to X$, the functor $f^*:C_\grp(X)\to C_\grp(X')$ takes $\ep{C}(X)$ to $\ep{C}(X')$.
\end{enumerate}
Both of these can be tested at the level of homotopy categories.
These are \Cref{sta:hexp-subfunctor,lem:pi-conv}.
\end{proof}

\begin{lem}
\label{sta:coreflector}%
The inclusion $\ep{C}\subseteq C_\grp$ admits a global right adjoint $\Pi$ which is nonunital symmetric monoidal.
\end{lem}
\begin{proof}
Fix $X\in\sch$. The inclusion $\ep{C}(X)\into C_\grp(X)$ admits a right adjoint by \Cref{sta:coreflector}.
By~\cite[Proposition~2.2.1.1]{HA}, it is nonunital lax symmetric monoidal.
It now suffices to show that the canonical morphism $\Pi(M)\conv\Pi(N)\to\Pi(M\conv N)$ is an equivalence for each $M,N\in C_{\grp}(X)$.
This was proven in \Cref{sta:conv-on-hCexp}.
\end{proof}

\begin{lem}
\label{sta:Cexp-unital}%
The association $X\mapsto \ep{C}(X)$ with the convolution product underlies a functor $\ep{C}:\schop\to\calg{\Cat{st}}$.
\end{lem}
\begin{proof}
We first show that it underlies a functor $\ep{C}:\schop\to\calg{\iCat}$.
By~\cite[Corollary~5.4.4.7]{HA}, it suffices to prove that $\ep{C}$ factors through $\calgqu{\iCat}$.
In other words, it suffices to verify the following two statements:
\begin{enumerate}
\item for $X\in\sch$, the homotopy category of $\ep{C}(X)$ with the convolution product has a unit,
\item for $f:X'\to X$, the functor $f^*:\ep{C}(X)\to\ep{C}(X')$ at the level of homotopy categories takes units to units.
\end{enumerate}
Both of these were established in \Cref{sec:exp-trcosy}, specifically the first in \Cref{sta:conv-on-hCexp} and the second in \Cref{sta:hCexp-cs}.

It remains to show that this functor actually factors through $\calg{\Cat{st}}$.
This amounts to the following two conditions:
\begin{enumerate}[resume]
\item the underlying functor $\ep{C}:\schop\to\iCat$ factors through $\Cat{st}$,
\item the convolution product on $\ep{C}(X)$ is biexact.
\end{enumerate}
By definition, $\ep{C}(X)$ is a stable \subicat of $C_{\grp}(X)$ and we already remarked in \Cref{sta:hexp-subfunctor} that the functor $\ep{C}(f):\ep{C}(X)\to\ep{C}(X')$ is the restriction of the exact functor $f_{\grp}:C_{\grp}(X)\to C_{\grp}(X')$ hence is exact itself.
The second condition is clear by the description of the convolution product in \Cref{sta:convolution-description}.
\end{proof}

\begin{thrm}
\label{sta:Cexp-cosy}%
Let $C$ be a (resp.\ presentable, compactly generated) coefficient system.
The functor $\ep{C}:\schop\to\calg{\Cat{st}}$ is a (resp.\ presentable, compactly generated) coefficient system.
Moreover, $\Pi:C_{\grp}^{\conv}\to\ep{C}$ is a morphism of (resp.\ presentable, compactly generated) coefficient systems.
\end{thrm}
\begin{proof}
The functor $\ep{C}$ was given in \Cref{sta:Cexp-unital}.
We noted in \Cref{sta:cosy-vs-triangulated} that the axioms may be verified at the level of (triangulated) homotopy categories.
The first statement (without parentheses) then follows from \Cref{sta:hCexp-cs}.
If $C$ is presentable and $X\in\sch[\base]$ then $C_{\grp}(X)$ is a presentable \icat.
By the recollement of \Cref{sta:exp-recollement-*}---which also exists at the level of \icats---we see that $\ep{C}(X)\simeq\ker(\pi_*)$ and the latter is an accessible localization of $C_{\grp}(X)$ hence itself presentable.
By \Cref{sta:hexp-subfunctor}, it is clear that for $f:Y\to X$, the functor $f^*:\ep{C}(X)\to\ep{C}(Y)$ preserves colimits since the same is true for $f^*:C_{\grp}(X)\to C_{\grp}(Y)$.
Finally, every functor that makes up the convolution product in \Cref{sta:convolution-description} preserves colimits at the level of $C_{\grp}$.
Therefore the convolution product preserves colimits in both variables separately.
A similar argument establishes the compactly generated case.

For the second statement, we have a natural transformation $\Pi:C_{\grp}^{\conv}\to\ep{C}$ of functors $\schop\to\calgnu{\Cat{st}}$ (\Cref{sta:coreflector}) and by~\cite[Corollary~5.4.4.7]{HA} and \Cref{sta:conv-on-hCexp}, it lifts to a natural transformation of functors $\schop\to\calg{\Cat{st}}$.
Finally, it defines a morphism of coefficient systems, by \Cref{sta:hCexp-cs} and \Cref{sta:cosy-vs-triangulated}.
The parenthesized statements follow from the (co)fiber sequence~\eqref{eq:triangle-Pi}.
\end{proof}

\begin{rmk}\label{sta:C-Cexp}
We expect that the functor $\Pi \unit_{*}:C(-)\to \ep{C}(-)$ underlies a morphism of coefficient systems, as in \cref{sta:D-Dexp} in the triangulated case. We will come back to this expectation in future work.
\end{rmk}

\begin{rmk}
\label{rmk:exp-functoriality}%
The constructions performed in this section are functorial in the following sense.
If $\phi:C\to C'$ is a morphism of coefficient systems then there is an induced morphism $\ep{\phi}:\ep{C}\to\ep{C'}$.
This can be proved by carrying the data of $\phi$ along in each step of the construction, using the theory of \emph{generalized} \ioperads~\cite[\S\,2.3.2]{HA}.
We will leave the construction of a functor
\[
\ep{(-)}:\CoSy{\base}\to\CoSy{\base}
\]
to the interested reader.
\end{rmk}

\section{Motives of varieties with potential}

Recall that a \emph{$\base$-scheme with potential} $(X,\pot)$ is simply an object of the category $\sch[\grp]$ seen as a pair $X\to \base$ and a $\base$-morphism $\pot:X\to \grp$. We often identify $\sch[\base]$ as a full subcategory of $\sch[\grp]$ via the functor $X\mapsto (X,\unit)$.

Unlike the exponential cohomology theories described in the introduction, the definition of $\ep{C}$ does not seem directly related to schemes with potential. In this section, we explain how to construct ``exponential motives'' in $\ep{C}$ associated to $\sch[\base]$-schemes with potential.

\subsection{Exponential twists}
\label{sec:exponential-twists}

The basic ingredients going into the construction of motives associated with schemes with potential are the `exponential twists' that we introduce in this subsection. They are of independent interest as well, as the analogues, in general exponentiated coefficient systems, of the exponential $\cD$-modules $\cE^{\pot}$ attached to a function $\pot:X\to \Ga{}$.
We refer to \Cref{sec:intro} for the motivation.

\begin{defn}\label{def:exp-twists}%
Let $(X,\pot)\in\sch[\grp]$. Write $z(\pot):X\to X_{\grp}=X\times_{\base}\grp$ for the closed immersion of the graph of $\pot$, and $u(\pot)$ for the open immersion of the complement. The \emph{exponential twist functor} associated to $(X,\pot)$ is given by
\begin{align*}
  \E{\pot}(-):C(X)&\to \ep{C}(X)\\
  M&\mapsto\Pi z(\pot)_{*}M
\end{align*}
We also write $\E{\pot}:=\E{\pot}(\one)$. 
\end{defn}

\begin{exa}
For $\pot=0$, this is the functor considered in \Cref{sta:D-Dexp}.
\end{exa}

\begin{lem}
\label{sta:conv-E_0()}
Let $X\in\sch[\base]$, let $M\in C_\grp(X)$, and $N\in C(X)$.
The K\"unneth formula and the counit $\Pi\to\id$ induce an equivalence in $\ep{C}(X)$:
\[
M\conv\E{\unit}(N)\simeq \Pi(M\otimes \pi^{*}N)
\]
\end{lem}
\begin{proof}
This is the composite:
\begin{align*}
M\conv\Pi\unit_*N
  &\simeq \Pi(M\conv \unit_*N)&&\text{\Cref{lem:Pi-conv}}\\
  &\simeq \Pi(M\otimes\pi^*N)&&\text{\Cref{sta:conv-0_*}}
\end{align*}
\end{proof}

Note that this recovers the fact (already known from \Cref{sta:conv-on-hCexp}) that $\E{\unit}=\Pi\unit_*\one\in\ep{C}(X)$ is a unit for the convolution product.

\begin{lem}
\label{sta:pullback-graph}
Let $f:(Y,\pot f)\to (X,\pot)$ be a morphism of $\grp$-schemes.
Then the following diagram in $\sch[\base]$ is Cartesian.
\[
\begin{tikzcd}
Y
\ar[r, "z(\pot f)"]
\ar[d, "f" left]
&
Y_{\grp}
\ar[d, "{f_{\grp}}"]
\\
X
\ar[r, "z(\pot)"]
&
X_{\grp}
\end{tikzcd}
\]
\end{lem}
\begin{proof}

This is straightforward.
\end{proof}

\begin{lem}\label{sta:exp-twists}
\begin{enumerate}[(a)]
\item\label{exp-twists-pullback} Let $f:(Y,\pot f)\to (X,\pot)$ be a morphism in $\sch[\grp]$. There is an equivalence of functors $C(X)\to \ep{C}(Y)$,
  \[
\ul{f}^{*}\circ\E{\pot}(-)\simeq \E{\pot f}(-)\circ f^{*}.
\]
\item\label{exp-twists-explicit} There is an equivalence of functors $C(X)\to\ep{C}(X)$,
  \[
  \E{\pot}(-)\simeq u(\pot)_{*}u(\pot)^{*}\pi^{!}(-)[-1].
  \]
  In particular, writing $u:=u(\unit)$, we have
  \[
\E{0}\simeq u_{*}u^{*}\pi^{!}\one[-1].
\]
\end{enumerate}  
\end{lem}
\begin{proof}
Claim \ref{exp-twists-pullback} follows directly from the definition together with \Cref{Pi-operations} and \Cref{sta:pullback-graph}.

By definition and \Cref{sta:Pi-coreflector}, we know that $\Pi z(\pot)_{*}$ is the fiber of $z(\pot)_{*}\to \pi^{!}\pi_{!}z(\pot)_{*}$ of the unit of adjunction.
It therefore suffices to prove that $u(\pot)_{*}u(\pot)^{*}\pi^{!}(-)$ is a cofiber of the same map.
Under the identifications $\pi_{!}z(\pot)_{*}\simeq (\pi z(\pot))_{!}=(\id)_{!}\simeq \id$ and $z(\pot)^{!}\pi^{!}\simeq \id$, this map is given by
\[
z(\pot)_{*}z(\pot)^{!}\pi^{!}\to \pi^{!},
\]
whose cofiber is as claimed, by \ref{ax:loc} for the coefficient system $C$.
This proves claim~\ref{exp-twists-explicit}.
\end{proof}

\begin{cor}
\label{sta:coprojector-as-convolution}
We have a natural equivalence of functors $C_\grp(X)\to\ep{C}(X)$:
\[
\Pi\simeq u_{*}u^{*}\pi^{!}\one[-1]\conv(-)
\]
\end{cor}
\begin{proof}
This follows from \cref{sta:conv-E_0()} applied with $N=\one$ and \cref{sta:exp-twists}\ref{exp-twists-explicit}.
\end{proof}

\begin{cor}
\label{sta:pullback-exp-kernel}
Let $\pot:X\to\grp$ be a potential.
Then we have in $\ep{C}(X)$:
\[
\E{\pot}\simeq\ul{\pot}^*\E{\id}.
\]
\end{cor}

\begin{rmk}
The object $\E{\id}\in\ep{C}(\grp)$ thus plays a special role in the theory; it generates all exponential twists by pullback, and as we will see in \cref{sec:motives} these generate motives of varieties with potential via the six-functor formalism of $\ep{C}$. We call $\E{\id}$ the \emph{exponential kernel} of $\ep{C}$. In the case where $\grp=\Ga{}$, it turns out to be a \emph{character sheaf} in $\ep{C}(\Ga)$ and to lead directly to a version of the Fourier transform, as discussed in the introduction.
\end{rmk}

\begin{notn}
\label{notn:addition-potentials}
Let $(X,\pot),(X',\pot')\in\sch[\grp]$.
We define the exterior sum of the two potentials as a new potential on $X\times_{\base}X'$:
\[
\pot\boxplus \pot':X\times_{\base}X'\xto{\pot\times\pot'}\grp\times_\base\grp\xto{\plus}\grp
\]
When $X=X'$ as $\base$-schemes we also define the sum:
\[
\pot + \pot':X\xto{\Delta}X\times_{\base}X'\xto{\pot\times\pot'}\grp\times_\base\grp\xto{\plus}\grp
\]
This endows the set $\Hom_{\base}(X,\grp)$ with the structure of an abelian group.
\end{notn}

\begin{lem}
\label{sta:Epot-product}
Let $(X,\pot),(X',\pot')\in\sch[\grp]$.
Then there is a canonical equivalence of functors $C(X)\times C(X')\to \ep{C}(X\times_{\base} X')$:
\[
\E{\pot}(-)\boxtimes\E{\pot'}(-)\simeq \E{\pot\boxplus\pot'}(-\boxtimes -)
\]
\end{lem}
\begin{proof}
Let us write $\pi_1:X\times_{\base}X'\to X$ and $\pi_2:X\times_{\base}X'\to X'$. Let $M\in C(X)$ and $M'\in C(X')$. By \Cref{sta:exp-twists}, we have an equivalence
\[
\E{\pot}(M)\boxtimes\E{\pot'}(M')=
\ul{\pi}_1^*\E{\pot}(M)\conv\ul{\pi}_2^*\E{\pot'}(M')\simeq
\E{\pot\pi_1}(\pi^*_1M)\conv\E{\pot'\pi_2}(\pi_2^*M').
\]
So the claim follows from the next \Cref{sta:Epot-product-internal}.
\end{proof}

\begin{lem}
\label{sta:Epot-product-internal}
Let $\pot,\pot':X\to \grp$ be two potentials.
Then there is a canonical equivalence of functors $C(X)\times C(X)\to \ep{C}(X)$:
\[
\E{\pot}(-)\conv\E{\pot'}(-)\simeq \E{\pot+\pot'}(-\otimes-)
\]
\end{lem}
\begin{proof}
Let $M,N\in C(X)$. We provide the following string of canonical equivalences:
\begin{align*}
    \E{\pot}\conv\E{\pot'}&=z(\pot)_*M\conv \E{\pot'}(N)\\
  &\simeq\plus_!
    \left(
    p_1^*z(\pot)_*M\otimes p_2^*\E{\pot'}(N)
    \right)\\
  &\simeq\plus_!
    \left(
    (z(\pot)\times\grp)_*\pi^{*}M\otimes p_2^*\E{\pot'}(N)
    \right)&&\text{Linearity (\Cref{cs-linearity})}\\
  &\simeq\plus_!(z(\pot)\times\grp)_*(\pi^{*}M\otimes (z(\pot)\times\grp)^* p_2^*\E{\pot'}(N))&&\text{Linearity (\Cref{cs-linearity})}\\
  &\simeq(\plus \circ (z(\pot)\times\grp))_!(\pi^{*}M\otimes \Pi z(\pot')_*N)&&p_2\circ (z(\pot)\times\grp)=\id\\
  &\simeq(\plus \circ (z(\pot)\times\grp))_!\Pi(\pi^{*}M\otimes z(\pot')_*N)&&\text{\Cref{lem:Pi-const}}\\
  &\simeq\Pi(\plus \circ (z(\pot)\times\grp))_! (\pi^{*}M\otimes z(\pot')_*N))&&\text{\Cref{Pi-operations}}\\
  &\simeq\Pi(\plus \circ (z(\pot)\times\grp\circ z(\pot'))_!(M\otimes N)&& \text{Linearity (\Cref{cs-linearity})}\\
  &\simeq\Pi z(\pot+\pot')_*(M\otimes N)\\
  &=\E{\pot+\pot'}(M\otimes N)
\end{align*}
\end{proof}

\begin{cor}
\label{sta:E-to-Pic}%
For each $\pot:X\to\grp$, the object $\E{\pot}$ is $\otimes$-invertible, and the map $\E{-}$ defines a group homomorphism
\[
\E{-}:\Hom_\base(X,\grp)\to\Pic(\hep{C}(X)).
\]
\end{cor}
\begin{proof}
We already remarked that the unit $\unit:X\to \grp$ is taken to the unit $\E{\unit}$ for the convolution product.
For $\pot,\pot':X\to \grp$ we have by \Cref{sta:Epot-product,sta:pullback-exp-kernel}:
\[
\E{\pot}\conv\E{\pot'}\simeq\ul{\Delta}^*(\E{\pot}\boxtimes\E{\pot'})\simeq\ul{\Delta}^*\E{\pot\boxplus\pot'}\simeq\E{(\pot\boxplus\pot')\circ\Delta}=\E{\pot+\pot'}
\]
Since $\pot+(-\pot)=\unit$, this shows both claims.
\end{proof}

\begin{rmk}
The maps $\E{-}$ of \Cref{sta:E-to-Pic} are part of a natural transformation
\[
\E{-}:\Hom_\base(-,\grp)\to\Pic(\hep{C}(-))
\]
of presheaves on $\base$-schemes.
This follows from \Cref{sta:exp-twists}.
\end{rmk}

\subsection{Motives}
\label{sec:motives}
Let $f:X\to Y$ be a morphism of $\base$-schemes.
Recall that that in any coefficient system~$C$, the ``(homological) motive'' associated with~$X$ can be defined in terms of the exceptional functors, like so: $\M(X):=f_!f^!\one\in C(Y)$.
The motive associated with a variety with potential $(X,\pot)$ will follow this idea, with a(n exponential) twist:

\begin{defn}\label{defn:motive}
The \emph{motive} of the variety with potential $(f:X\to Y,\pot:X\to \grp)$ is defined to be 
\[
\M(X,\pot):=\M(X,f,\pot):=\ul{f}_!(\E{-\pot}\otimes \ul{f}^!\one)\in C(Y).
\]
More generally, if $g:Z\into X$ is a subscheme then the \emph{relative motive} of $(X,Z,\pot)$ is defined to be
\[
\M(X,Z,\pot):=\cofib\left(\M(Z,\pot\circ g)\to \M(X,\pot)\right)
\]
where the map is given by the composition
\begin{align}
  \label{eq:relative-motive-map}
  \ul{f}_!\ul{g}_!(\E{-\pot g}\otimes \ul{g}^!\ul{f}^!\one)&\simeq \ul{f}_!\ul{g}_!(\ul{g}^*\E{-\pot}\otimes \ul{g}^!\ul{f}^!\one)&&\text{\Cref{sta:exp-twists}}\\
  \notag{}&\isoto \ul{f}_!(\E{-\pot}\otimes \ul{g}_!\ul{g}^!\ul{f}^!\one)&&\text{projection formula}\\
  \notag{}&\to \ul{f}_!(\E{-\pot}\otimes \ul{f}^!\one)&&\text{counit of adjunction}
\end{align}
\end{defn}

\begin{rmk}\label{rmk:bivariant}
The motive of $(X,\pot)$ in \cref{defn:motive} is more precisely the \emph{homological} motive of the variety with potential. In fact, this is one of four motives attached to $(X,\pot)$, as in the classical case:
\begin{center}
\begin{tabular}{ccrcl}
  homological motive && $M(X,a):=f_!(\E{-\pot}\otimes f^!\one)$\\
  cohomological motive && $M^{\coh}(X,\pot):=f_*(\E{\pot}\otimes f^*\one)$\\
  Borel-Moore homological motive && $M^{\BM}(X,a):=f_*(\E{-\pot}\otimes f^!\one)$ \\
  cohomological motive with compact support && $M^{\coh}_{c}(X,\pot):=f_!(\E{\pot}\otimes f^*\one)$
\end{tabular}
\end{center}
These exhibit a similar behaviour as in the classical setting, with the appropriate functoriality with respect to proper, \'etale, \ldots{} maps. We will not develop this here. The signs in $\E{\pm \pot}$ are not arbitrary; they are compatible with the general definitions of bivariant theories in motivic homotopy theory of~\cite[Definition 2.2.1]{deglise-jin-khan:fundamental-classes}, and are also fixed by the requirement that exponential realization functors send $M(X,\pot)$ (resp. $M^{\coh}(X,\pot)$,\ldots{}) to the corresponding ``exponential (co)homology groups'' (e.g rapid decay homology, resp. cohomology, \ldots{} in the Betti case). 
\end{rmk}

\begin{lem}
\label{sta:motive-functorial}
Let $g:(Z,\pot g)\to (X,\pot)$ be a morphism (not necessarily a closed immersion) of varieties with potential.
Then the morphism
\[
\M(g):\M(Z,\pot g)\to\M(X,\pot)
\]
of~\eqref{eq:relative-motive-map} is functorial in $g$.
\end{lem}
\begin{proof}
This is true for each of the morphisms in the composition~\eqref{eq:relative-motive-map}.
\end{proof}

\begin{lem}\label{sta:motive-exp-smooth}
Let $(X,f,\pot)$ be a variety with potential and assume $f$ is smooth.
Then
\[
\M(X,\pot)\simeq \ul{f}_\sharp\E{-\pot}.
\]
\end{lem}
\begin{proof}
Indeed,
\[
\ul{f}_\sharp\E{-\pot}\simeq \ul{f}_\sharp(\E{-\pot}\otimes \ul{f}^*\one)\simeq \ul{f}_!(\E{-\pot}\otimes \ul{f}^!\one)
\]
by relative purity.
\end{proof}

For the next result we need some preparations.
\begin{lem}\label{*!-smooth}
Let $fg=hk$ be a Cartesian square with one of $f$ or $h$ smooth. Then the exchange morphism
  \begin{equation*}
    g^*f^!\to k^!h^*
  \end{equation*}
is invertible in any coefficient system.
\end{lem}
\begin{proof}
  Say $f$ is smooth so that $f^!\simeq \thom(\Omega_f)f^*$. Then
  \begin{align*}
    g^*f^!\simeq g^*\thom(\Omega_f)f^*\simeq \thom(g^*\Omega_f)g^*f^*\simeq\thom(\Omega_k)k^*h^*\simeq k^!h^*
  \end{align*}
  which is the exchange morphism~\cite[Proposition~1.5.19]{ayoub:thesis-1}.
\end{proof}

\begin{lem}\label{dualizing-boxtimes}
Let $f:X\to \base$ and $f':X'\to \base$ be two smooth morphisms.
There is a canonical isomorphism $f^!\one\boxtimes (f')^!\one\simeq (f\times f')^!\one$ in any coefficient system.
\end{lem}
\begin{proof}
We have 
\begin{align*}
  f^!\one\boxtimes (f')^!\one & \simeq      (\id_X\times f')^*f^!\one\otimes (f\times\id_{X'})^*(f')^!\one\\
    &\simeq (\id_X\times f')^*f^!\one\otimes (\id_X\times f')^!f^*\one && \text{\cref{*!-smooth}}\\
    &\simeq (\id_X\times f')^!(f^!\one\otimes f^*\one) && \text{\cref{*!-smooth}}\\
    &\simeq (\id_X\times f')^!f^!\one\\
    &\simeq (f\times f')^!\one.
\end{align*}
where we have used that, for a smooth morphism $g$, we have $g^{!}(M\otimes N)\simeq g^{*}(M)\otimes g^{!}(N)$ as consequence of relative purity.
\end{proof}

\begin{prop}
\label{sta:Kuenneth-homology}
Let $(X,\pot),(X',\pot')\in\sch[Y\times\grp]$.
Then in $\ep{C}(Y)$ we have a morphism
\[
\M(X,\pot)\otimes\M(X',\pot')\to\M(X\times_{Y} X',\pot\conv\pot')
\]
which is invertible if $X$ and $X'$ are smooth $Y$-schemes.
\end{prop}
\begin{proof}
This is the following composite:
\begin{align*}
  \ul{f}_!(\E{\pot}\otimes \ul{f}^!\one)\otimes \ul{f}'_!(\E{\pot'}\otimes (\ul{f}')^!\one)&\simeq \ul{f\times f'}_!\left((\E{\pot}\otimes \ul{f}^!\one)\boxtimes (\E{\pot'}\otimes (\ul{f}')^!\one)\right)\\
  &\simeq \ul{(f\times f')}_!\left(\E{\pot\conv \pot'}\otimes (\ul{f}^!\one\boxtimes \ul{(f')}^!\one)\right)\\
&\to  \ul{(f\times f')}_!\left(\E{\pot\conv\pot'}\otimes \ul{(f\times f')}^!\one\right)
\end{align*}
Here, the first equivalence is linearity, the second is \Cref{sta:Epot-product}, and the third is the canonical isomorphism of \Cref{dualizing-boxtimes}.
\end{proof}

\subsection{Generators}
\label{sec:generators}

\begin{defn}
Let $\cE$ be a collection $\cE_X\subset \Pic(C(X))$ of invertible objects, for each $X\in\sch[\base]$, satisfying $g^*\cE_Y\subseteq \cE_X$ for each $g:X\to Y$ and $\cE_{X}(n)=\cE_{X}$ for each $n\in \Z$.
We say that $C$ is \emph{Ind-$\cE$-constructible} if for each $Y\in\sch[\base]$, the stable \icat $C(Y)$ is the smallest localizing subcategory containing the motives $f_\sharp E$ for all $f:X\to Y$ smooth and $E\in \cE_X$.
\end{defn}

\begin{exa}
For example, if $\cE_X=\one(\Z):=\{\one(n)\mid n\in\Z\}$ for all $X$, this is known as of \emph{geometric origin}, see~\cite[\S\,2.4.3]{gallauer:six-functor-survey} for more details on this notion.
An important example to have in mind is $\SH$, stable motivic homotopy theory.

More generally, if $\Lambda\subset\Pic(C(\base))$ is some class containing all Tate twists and we set $\cE_X=f^*\Lambda$ for each $f:X\to\base$ then the resulting notion is closely related to \emph{$\Lambda$-constructibility} as defined in~\cite[\S\,2.3.10]{ayoub:thesis-1}.
(There, $\Lambda$ is allowed to consist of non-invertible objects as well.)
\end{exa}

\begin{prop}
\label{sta:smooth-generators}
If $C$ is Ind-$\cE$-constructible then $\ep{C}$ is Ind-$\cE_{\exp{}}$-constructible, where
\[
\cE_{\exp{},Y}=\{\E{\pot}(E)\mid E\in\cE_Y, a:X\to\grp\}.\footnote{As seen in the proof, one may restrict oneself to potentials $a$ that factor as a smooth morphism $b:X\to Y\times\grp$ followed by the canonical projection. Thus if $Y$ is smooth $a$ may be assumed to be smooth too.}
\]
In particular, if $C$ is of geometric origin then $\ep{C}(Y)$ has generators
\[
\M(X,\pot)(n),\qquad \pot:X\to \grp, n\in\Z.
\]
\end{prop}
\begin{proof}
By assumption, $C_\grp(Y)$ is generated by motives of the form $b_{\sharp}E$ where $b:X\to Y\times\grp$ is smooth and $E\in\cE_X$.
As $\ep{C}(Y)$ is a Verdier colocalization of $C_\grp(Y)$ with colocalization functor $\Pi$ we deduce that the former is generated by objects of the form $\Pi(b_\sharp E)$.
By definition of~$\Pi$ (~\cref{sta:Pi-coreflector}), we need to identify the fiber of the morphism $b_\sharp E\to \pi^*\pi_\sharp b_\sharp E$.
Let $f=\pi\circ b:X\to Y$ and $a=\pi_Y\circ b:X\to\grp$ be the two components of~$b$.
By a Mayer-Vietoris argument we may pass to a small enough Zariski open cover of $X$ so that the vector bundles $\Omega_f$ and $\Omega_b$ are trivial hence their Thom ``spaces'' $\thom(\Omega_f)$ and $\thom(\Omega_b)$ are Tate twists (and shifts).
It then follows that up to shifts, the object $E'=\thom(\Omega_b)\thom(\Omega_f)^{-1}E$ also belongs to $\cE_X$.
We need to identify the fiber of the morphism
\begin{equation}
\label{eq:unit-exp-motive1}
b_!(\thom(\Omega_f)E')\to\pi^*\pi_\sharp b_!(\thom(\Omega_f)E')
\end{equation}
in $C_\grp(Y)$.

Now, $\E{\pot}(E')$ may be written as the fiber of the unit morphism
\[
z(\pot)_*E' \to \pi^*\pi_\sharp z(\pot)_*E'
\]
from which one easily deduces that the object $\ul{f}_\sharp\E{\pot}(E')$ in $\ep{C}(Y)$ is the fiber of the unit morphism
\begin{equation}
\label{eq:unit-exp-motive2}
(f_\grp)_\sharp z(\pot)_*E' \to \pi^*\pi_\sharp(f_\grp)_\sharp z(\pot)_*E'
\end{equation}
in $C_\grp(Y)$.
Using purity again in the form of $(f_\grp)_\sharp=(f_\grp)_!\thom(\Omega_{f_\grp})=(f_\grp)_!\thom(\pi^*\Omega_f)$
we identify~\eqref{eq:unit-exp-motive2} with
\[
(f_\grp)_!z(\pot)_*\thom(\Omega_f)E'\to\pi^*\pi_\sharp(f_\grp)_!z(\pot)_*\thom(\Omega_f)E'.
\]
Finally, since $f_\grp\circ z(\pot)=b$ we get precisely~\eqref{eq:unit-exp-motive1}.
\end{proof}

\subsection{Exponential motivic cohomology}
\label{sec:motivic-cohomology}

We also have an associated cohomology for varieties with potential. For the case $C=\DM(-)$ which is one of our main interests, we want to think of these as exponential motivic cohomology; the name is not really appropriate in other cases, but we adopt it in this section.

\begin{defn}
  Let $(X,\pot)\in\sch[\base\times\grp]$ and $v\in K_{0}(X)$. We define the corresponding \emph{exponential motivic cohomology spectrum} of $(X,\pot)$ twisted by $v$ as the mapping spectrum
  \[
\Hm_{\exp}(X,\pot;v):=\mapsp_{\ep{C}(X)}(\E{0},\thomexp(v)( \E{\pot})).
  \]
We define for $p,q\in\Z$ the \emph{exponential motivic cohomology groups} of $(X,f)$ as
  \[
\Hm^{p}_{\exp}(X,\pot;\Z(q)):=\Hom_{\ep{C}(X)}(\E{0},\E{\pot}(q)[p]).
\]
We use correspondingly the notations
\[
\Hm_{\mot}(X,v):=\mapsp_{C(X)}(\one,\thom(v)(\one))
\]
and
\[
\Hm^{p}_{\mot}(X,\Z(q)):=\Hom_{C(X)}(\one,\one(q)[p])
\]
for ``ordinary'' cohomology for $C$. If necessary, we decorate these notations by $\ep{C}$ or $C$ respectively to indicate which coefficient system this refers to.
\end{defn}

\begin{lem}\label{lem:mot-coh-smooth}
Let $(X,\pot)\in\sch[\base\times\grp]$ and assume that $f:X\to \base$ is smooth. Let $w\in K_{0}(\base)$. We then have
\[
\Hm_{\exp}(X,\pot;f^{*}w)\simeq \mapsp_{\ep{C}(\base)}(M(X,\pot),\thomexp(w)).
\]
\end{lem}
\begin{proof}
This follows from \ref{sta:motive-exp-smooth} and adjunction.  
\end{proof}

We now express the exponential motivic cohomology groups of varieties with potentials in terms of ``ordinary'' motivic cohomology/morphism groups in~$C$.

\begin{prop}\label{prop:exp-mot-coh-comp}
Let $(X,\pot)\in\sch[\base\times\grp]$ and $v\in K_{0}(X)$. Let $V(\pot):=\pot^{-1}(\unit)\stackrel{i(\pot)}{\hookrightarrow}X$ be the vanishing locus of $\pot$ and $\cC_{\base/\grp}=\unit^*\Omega_{\grp/\base}$ the conormal sheaf (pulled back to~$X$).
We have a fiber sequence of spectra
  \[
 \Hm_{\mot}(X,v-\cC_{\base/\grp})\to \Hm_{\mot,V(\pot)}(X,v)\to \Hm_{\exp}(X,\pot;v)
 \]
where $\Hm_{\mot,V(\pot)}(X,v)=\mapsp_{C(V(\pot))}(\one,i(\pot)^{!}\thom(v))$ is motivic cohomology with supports. In particular, for $\grp=\Ga{}$ we have a long exact sequence
\[\cdots\to \Hm^{p}_{\mot,V(\pot)}(X,\Z(q)) \to \Hm^{p}_{\exp}(X,\pot;\Z(q)) \to \Hm^{p-1}_{\mot}(X,\Z(q-1))\to\Hm^{p+1}_{\mot,V(\pot)}(X,\Z(q))\to\cdots\]
\end{prop}  
\begin{proof}
We have
\begin{eqnarray*}
  \Hm_{\exp}(X,\pot;v) &\stackrel{\mathrm{def}}{=}& \mapsp(\E{\unit},\thomexp(v)( \E{\pot}))\\
                       &\simeq & \mapsp(\E{\unit},\E{\pot}(\one)\conv \E{\unit}(\thom(v)))\\
                       & \simeq & \mapsp(\E{\unit},\E{\pot}(\thom(v)))\\
  & \simeq & \mapsp(\E{\unit},z(\pot)_{*}\thom(v))
\end{eqnarray*}
where the second isomorphism follows from the definition of $\E{\pot}$ and \cref{lem:thom-exp}, the third from \cref{sta:Epot-product-internal}, and the fourth from the fact that $\Pi$ is a colocalisation.

By \cref{sta:Pi-coreflector}, we have a fiber sequence
\[
\E{\unit}\to \unit_{*}\one \to \pi^{!}\pi_{!}\unit_{*}\one\simeq \pi^{!}\one\simeq \pi^*\thom{\cC_{\base/\grp}}
\]
so we get an induced fiber sequence
\[
\mapsp(\pi^*\thom(\cC_{\base/\grp}),z(\pot)_{*}\thom(v))\to \mapsp(\unit_{*}\one,z(\pot)_{*}\thom(v))\to \Hm_{\exp}(X,\pot;v).
\]
By adjunctions and base change, this is equivalent to
\[
\mapsp(\one,\thom(v-\cC_{\base/\grp}))\to\mapsp(\one,i(\pot)^{!}\thom(v))\to \Hm_{\exp}(X,\pot;v)
\]
which finishes the proof.
\end{proof}

In \Cref{sta:D-Dexp} we showed that the functor $C(X)\to\ep{C}(X)$ admits a retract in many cases.
However, the functor is not fully faithful in general:
\begin{cor}\label{cor:mot-coh-unit}
Let $X\in\sch[\base]$ and $v\in K_{0}(X)$. Then we have a fiber sequence
  \[
\Hm(X,v-\cC_{\base/\grp})\to \Hm(X,v)\to \Hm_{\exp}(X,\unit;v).
\]
The functor $\Pi \unit_*:C(X)\to \ep{C}(X)$ is not fully faithful in general. 
\end{cor}
\begin{proof}
The first claim is just the special case of \cref{prop:exp-mot-coh-comp} where $\pot=\unit$. Let $X\in \Sm{\base}$ and $v\in K_{0}(X)$. Then it is easy to see using \Cref{Pi-operations} that we have $\Pi\unit_{*}M(X)\simeq M(X,\unit)$ in $\ep{C}(\base)$. Hence by \cref{lem:mot-coh-smooth} and \cref{lem:thom-exp}, we have an equivalence
  \[
\mapsp(\Pi\unit_{*}M(X),\Pi\unit_{*}\mthom{v})\simeq \Hm_{\exp}(X,\unit;v)
\]
in $\Sp$. It is then easy to see from the proof of \cref{prop:exp-mot-coh-comp} that the morphism on mapping spectra induced by $\Pi \unit_{*}$ is precisely the second morphism in the fiber sequence
\[
\Hm(X,v-\cC_{\base/\grp})\to \Hm(X,v)\to \Hm_{\exp}(X,\unit;v).
\]
We conclude since the spectrum $\Hm(X,v-\cC_{\base/\grp})$ is non-trivial in general.
\end{proof}  

\begin{exa}
\label{exa:Euler-Mascheroni}%
Let $X=\base=\Spec(k)$ be a field and $C=\DM(-,\Q)$.
We then have an exact sequence
\[
0=\Hm^{-1}(k,\Q)\to \Hm^1(X,\Q(1))\to\Hm_{\exp}^1(X,0;\Q(1))\to\Hm^0(k,\Q)\to\Hm^2(k,\Q(1))=0
\]
and $\Hm(k,\Q)=\Q$ so that this gives an explicit example in which the functor $\Pi\unit_*$ is not full.
``This'' new exponential motive $M(\gamma)$ that sits in a fiber sequence
\[
\Q(1)\to M(\gamma)\to \Q
\]
is called in~\cite{fresan-jossen:expmot} the \emph{Euler-Mascheroni motive} as it has the Euler-Mascheroni constant~$\gamma$ as a period.
\end{exa}

\begin{defn}\label{defn:bs-vanishing}
Let $X\in\sch[\base]$. We say that $X$ satisfies the \emph{Beilinson-Soulé vanishing condition} $(\textbf{BS}_{C}(X))$ for $C$ if for all $p<0$ and $q\in\Z$, we have
  \[
\Hm^{p}_{\mot}(X,\Z(q))=0.
\]
\end{defn}

\begin{cor}\label{cor:exp-BS}
Let $\grp=\Ga{}$ and let $X\in \sch[\base]$. Then $(\textbf{BS}_{C}(X))$ implies $(\textbf{BS}_{\ep{C}}(X))$.
\end{cor}
\begin{proof}
The fiber sequence in \cref{cor:mot-coh-unit} gives a long exact sequence
  \[
\cdots\to \Hm^{p}_{\mot}(X,\Z(q))\to \Hm^{p}_{\exp}(X,\unit;\Z(q))\to \Hm^{p-1}_{\mot}(X,\Z(q-1))\to \Hm^{p+1}_{\mot}(X,\Z(q))\to\cdots
\]
which implies the claim.
\end{proof}  

\begin{rmk}\label{rmk:tate-mot}
Corollary \ref{cor:exp-BS} implies that when $C=\DM(-,\Q)$ and $X$ satisfies the classical Beilinson--Soulé conjecture, then the method of \cite{levine:tate} applies to the subcategory $\ep{\DTM}(X,\Q)$ of \emph{exponential mixed Tate motives} generated in $\ep{\DM}(X,\Q)$ by the Tate motives and produces a motivic t-structure. We plan to study in detail this t-structure in future work, as well as t-structures for stratified exponential mixed Tate motives as in \cite{soergel-wendt:perverse-tate}.
\end{rmk}

By a similar argument, we can compute certain mapping spaces between motives in $\ep{C}$.

\begin{prop}\label{prop:some-map-spaces}
  Let $(X,\pot),(Y,b)\in\sch[\base\times\grp]$ and $w\in K_{0}(\base)$. Assume that $X$ is smooth over $\base$ and that $b$ is proper. Let $V(-\pot\boxplus b):=\{(x,y)|-\pot(x)=b(y)\}\subset X\times_{\base}Y$ and $\cC_{\base/\grp}=\unit^*\Omega_{\grp/\base}$ the conormal sheaf. Then we have a fiber sequence of spectra
  \[
 \Hm(X\times_{\base}Y,w-\cC_{\base/\grp})\to \Hm_{V(-\pot\boxplus b)}(X\times_{B}Y,w)\to \mapsp(M(X,\pot),M^{\coh}_{c}(Y,b)\conv \thomexp(w)).
\]
\end{prop}
\begin{proof}
To simplify the notation, we only treat the case $w=0$. Write $f:X\to \base$ and $g:Y\to \base$ for the structure morphisms. Since $X$ is smooth over $\base$, we have $M(X,\pot)=\ul{f}_{\sharp}\E{-\pot}$ by \cref{sta:motive-exp-smooth}, and we compute:
  \begin{eqnarray*}
    \mapsp(M(X,\pot),M^{\coh}_{c}(Y,b)) & \simeq & \mapsp(\ul{f}_{\sharp}\E{-\pot},\ul{g}_{!}\E{b})\\
    & \simeq & \mapsp(\E{-\pot},\ul{f}^{*}\ul{g}_{!}\E{b}) \\
& \simeq & \mapsp(\E{0},\ul{p_{1}}_{!}\E{b p_{2}}\conv \E{\pot}) \\
    & \simeq & \mapsp(\E{0},\ul{p_{1}}_{!}(\E{\pot\boxplus b})) \\
     & \simeq & \mapsp(\E{0},(p_{1,\grp})_{!}\Pi z(\pot \boxplus b)_{!}\one) \\
    & \simeq & \mapsp(\E{0},(p_{1,\grp}\circ z(\pot \boxplus b))_{!}\one))
  \end{eqnarray*}
where besides adjunction and base change and linearity in $\ep{C}$, we have used \cref{sta:Epot-product-internal}, \cref{sta:exp-twists} and \cref{Pi-operations}. We now notice that we have
  \[
p_{1,\grp}\circ z(\pot \boxplus b)=(\id_{X}\times\,\plus)\circ (z(\pot)\times b)=(\id_{X}\times\,\plus)\circ (z(\pot)\times \id_{\grp})\circ (\id_{X}\times\,b).
\]
Moreover, $(\id_{X}\times\,\plus)\circ (z(\pot)\times \id_{\grp})$ is an automorphism of $X\times_{\base}\grp$, with inverse $(\id_{X}\times\,\plus)\circ (z(-\pot)\times \id_{\grp})$.
Since $z(\pot)\times \id_{\grp}$ is a closed immersion, this implies that
\[
((\id_{X}\times\,\plus)\circ (z(\pot)\times \id_{\grp}))_{!}\simeq ((\id_{X}\times\,\plus)\circ (z(\pot)\times \id_{\grp}))_{*}\simeq (\id_{X}\times\,\plus)_{*}\circ (z(\pot)\times \id_{\grp})_{!}.
\]
We conclude that
\[
\mapsp(\E{0},(p_{1,\grp}\circ z(\pot \boxplus b))_{!}\one))\simeq \mapsp(\E{0},(\id\times\,\plus)_{*}(z(\pot)\times b)_{!}\one).
\]
Since $b$ is proper, so is $(z(\pot)\times b)$, and we thus conclude that
\[
\mapsp(M(X,\pot),M^{\coh}_{c}(Y,b)) \simeq \mapsp(\E{0},((\id\times\,\plus)\circ(z(\pot)\times b))_{*}\one).
\]
and the same computation as in the proof of \Cref{prop:exp-mot-coh-comp} gives the result.
\end{proof}

\phantomsection
\addcontentsline{toc}{section}{References}
\bibliographystyle{alpha}
\bibliography{ref}

\end{document}